\documentclass[oneside,english]{amsart}
\usepackage[T1]{fontenc}
\usepackage[latin9]{inputenc}
\usepackage{xcolor}
\usepackage{babel}
\usepackage{mathrsfs}
\usepackage{amstext}
\usepackage{amsthm}
\usepackage{amssymb}
\usepackage[unicode=true,pdfusetitle,
 bookmarks=true,bookmarksnumbered=false,bookmarksopen=false,
 breaklinks=false,pdfborder={0 0 1},backref=false,colorlinks=false]
 {hyperref}

\makeatletter
\numberwithin{equation}{section}
\numberwithin{figure}{section}
\theoremstyle{plain}
\newtheorem{thm}{\protect\theoremname}
\theoremstyle{plain}
\newtheorem{lem}[thm]{\protect\lemmaname}
\theoremstyle{definition}
\newtheorem{defn}[thm]{\protect\definitionname}
\theoremstyle{definition}
\newtheorem{example}[thm]{\protect\examplename}
\theoremstyle{plain}
\newtheorem{prop}[thm]{\protect\propositionname}
\theoremstyle{plain}
\newtheorem{cor}[thm]{\protect\corollaryname}
\theoremstyle{remark}
\newtheorem{notation}[thm]{\protect\notationname}
\theoremstyle{remark}
\newtheorem{rem}[thm]{\protect\remarkname}

\usepackage[a4paper,top=1.055in, bottom=0.96in, left=0.88in, right=0.88in]{geometry} 
\usepackage{xypic}
\usepackage{xcolor}

\makeatother

\providecommand{\corollaryname}{Corollary}
\providecommand{\definitionname}{Definition}
\providecommand{\examplename}{Example}
\providecommand{\lemmaname}{Lemma}
\providecommand{\notationname}{Notation}
\providecommand{\propositionname}{Proposition}
\providecommand{\remarkname}{Remark}
\providecommand{\theoremname}{Theorem}

\begin{document}
\author{Adrien Dubouloz} 
\address{IMB UMR5584, CNRS, Univ. Bourgogne Franche-Comt\'e, F-21000 Dijon, France.} 
\email{adrien.dubouloz@u-bourgogne.fr}

\author{Takashi Kishimoto} 
\address{Department of Mathematics, Faculty of Science, Saitama University, Saitama 338-8570, Japan} 
\email{tkishimo@rimath.saitama-u.ac.jp}

\author{Pedro Montero} 
\address{Departamento de Matemática Universidad Técnica Federico Santa María \newline \indent Avenida España 1680,  Valparaíso, Chile}
\email{pedro.montero@usm.cl}  


\title[Del Pezzo quintics as equivariant compactifications of vector groups]{Del Pezzo quintics as equivariant \\ compactifications of vector groups} 
\begin{abstract}
We study faithful actions with a dense orbit of abelian unipotent
groups on quintic del Pezzo varieties over a field of characteristic
zero. Such varieties are forms of linear sections of the Grassmannian
of planes in a 5-dimensional vector space. We characterize which smooth
forms admit these types of actions and show that in case of existence,
the action is unique up to equivalence by automorphisms. We also give
a similar classification for mildly singular quintic del Pezzo threefolds
and surfaces. 
\end{abstract}

\maketitle
\vspace{-0.8cm}

\section*{Introduction}

Vector group varieties are defined by analogy to toric varieties as
being varieties $X$ endowed with an effective action of an abelian
unipotent group $\mathbb{U}\cong\mathbb{G}_{a}^{n}$ with a Zariski
dense open orbit. For varieties defined over a field of characteristic
zero, the group $\mathbb{U}$ then embeds equivariantly as the open
orbit, making $X$ into a partial equivariant completion of $\mathbb{U}$.
The study of such equivariant completions which are Fano varieties
was initiated by Hassett and Tschinkel in \cite{HT99} with views
towards Manin's conjecture on the asymptotic distribution of rational
points of bounded height over number fields (see e.g. \cite{CLT02,CLT12}).
Over algebraically closed fields of characteristic zero, besides projective
spaces of every dimension, many families of Fano varieties and other
Mori fibers spaces including for instance smooth projective quadrics,
Grassmannians and flag varieties are known to be vector group varieties,
see e.g. \cite{Arz11, AS11,HF20, FM19, HM20, Nag20,Sha09}. 

Of particular interest in this context is the question of classification
of possible equivalence classes of structures of vector group variety
up to isomorphisms on a given variety. Indeed, it was observed by
Hassett and Tschinkel \cite{HT99} that in contrast to toric structures,
vector group variety structures on $\mathbb{P}_{\mathbb{C}}^{n}$,
$n\geq2$, are not unique and that for $n\geq6$ there are even infinitely
many equivalence classes of such structures. In contrast, it is known
that over algebraically closed fields of characteristic zero, Grassmannians
other than projective spaces \cite{AS11} and smooth quadrics \cite{Sha09}
admit a unique equivalence class of structure of vector group variety.
For Fano varieties of Picard rank one over the field of complex numbers,
Fu and Hwang \cite{HF14} gave a uniform characterization of uniqueness
of vector group variety structures in terms in the smoothness of the
Variety of Minimal Rational Tangents (VMRT) of a dominating family
of minimal rational curves. 

In this article, we consider the problem of existence and uniqueness
of vector group variety structures on smooth and mildly singular quintic
del Pezzo varieties over an arbitrary field $k$ of characteristic
zero. By definition, a \emph{smooth quintic del Pezzo variety} is
a smooth geometrically connected projective $k$-variety $X$ whose
base extension $X_{\bar{k}}=X\times_{\mathrm{Spec}(k)}\mathrm{Spec}(\bar{k})$
to an algebraic closure $\bar{k}$ of $k$ has an ample invertible
sheaf $\mathcal{L}$ of degree $5$ such that $\omega_{X_{\bar{k}}}^{\vee}\cong\mathcal{L}^{\otimes(n-1)}$.
If $n=2$ then $X_{\bar{k}}$ is a del Pezzo surface of degree $5$.
On the other hand, by \cite{Fuj81}, a smooth quintic del Pezzo $k$-variety
$X$ exists if and only if $n\leq6$ and furthermore, for each $n=3,4,5,6$,
$X_{\bar{k}}$ is unique up to isomorphism, isomorphic to any smooth
linear section of the Grassmannian $\mathrm{G}(2,5)\subset\mathbb{P}_{\bar{k}}^{9}$
of $2$-dimensional vector subspaces of $\bar{k}^{\oplus5}$. For
$n\leq3$, the automorphism group of $X_{\bar{k}}$ is either finite
if $n=2$ or isomorphic to $\mathrm{PGL}_{2}(\bar{k})$ if $n=3$
(see e.g. \cite[Proposition 7.1.10]{CS1}). In particular, it cannot
contain any abelian unipotent subgroup defining a vector group variety
structure on $X$. Over the field of complex numbers, existence and
uniqueness of vector group variety structures on the remaining varieties,
which are all Fano of Picard rank $1$, has been settled affirmatively
by Fu and Hwang as a consequence of a series of articles \cite{HF18,HF20}
devoted to the broader study using VMRT techniques of so-called Euler-symmetric
projective varieties. 

Here, since we work over arbitrary fields $k$ of characteristic zero,
possibly non-closed, the actual question becomes to determine and
classify $k$-forms of vector group variety structures on del Pezzo
quintics. Of course, the non-existence of such structures after base
extension to an algebraic closure is a clear obstruction for existence
of these structures over the given base field, But on the other hand,
neither the existence nor its combination with uniqueness up to equivalence
after base extension is enough in general to conclude, say by Galois
descent arguments, the existence of such structures defined over the
base field. For instance, a smooth $n$-dimensional quadric in $\mathbb{P}_{\mathbb{Q}}^{n+1}$
without $\mathbb{Q}$-rational point does not admit any vector group
variety structure defined over $\mathbb{Q}$, even though its base
extension to $\bar{\mathbb{Q}}$ admits infinitely many such structures,
which, as a consequence of \cite{Sha09}, are all equivalent. Our
approach is thus by necessity different from that using VMRT techniques
in \cite{HF18,HF20,FM19}, which, in particular, depend on the Cartan-Fubini
analytic extension theorem \cite{HM01} that has no arguably straightforward
counterpart over arbitrary fields. Instead we build on elementary
birational geometry of the Grassmannian $\mathrm{G}(2,5)$, its linear
Schubert sub-varieties and their associated rational projections,
a material which, over algebraically closed fields, goes back to a
classical article of Todd \cite{Todd30}. Our first main result is
a classification of $k$-forms of vector group variety structures
on smooth del Pezzo quintics of dimension $\geq4$ which can be summarized
as follows: 
\begin{thm}
\label{thm:MainTh1}For a smooth quintic del Pezzo $k$-variety $X_{n}$
of dimension $n\in\{4,5,6\}$, the following hold: 

1) If $n=6$ then $X_{6}$ admits a vector group variety structure
if and only if it has a $k$-rational point. If so, $X_{6}$ is isomorphic
to $\mathrm{G}(2,5)\subset\mathbb{P}_{k}^{9}$ and it admits a unique
vector group variety structure up to equivalence.

2) If $n=4,5$ then $X_{n}$ is unique up to isomorphism, isomorphic
to any smooth section of $\mathrm{G}(2,5)\subset\mathbb{P}_{k}^{9}$
by a linear subspace of codimension $6-n$. Furthermore, it admits
precisely one vector group variety structure up to equivalence.
\end{thm}

In a second step, we apply the same methods to $k$-forms of vector
group variety structures on mildly singular del Pezzo quintic threefolds
and surfaces. The case of quintic del Pezzo surfaces with canonical
singularities has already been fully settled by Derenthal and Loughran
\cite{DL10} who studied vector group variety structures on del Pezzo
surfaces with canonical singularities over arbitrary fields of characteristic
zero. We therefore mainly focus on the case of quintic del Pezzo threefolds
with terminal singularities. We obtain the following characterization,
which says in particular that in contrast to smooth del Pezzo quintics
of dimension $4$ and $5$ and canonical del Pezzo quintic surfaces
which have no non-trivial $k$-forms for any field $k$, trinodal
del Pezzo quintics threefolds do in general have non-trivial $k$-forms:
\begin{thm}
\label{thm:MainTh2}A quintic del Pezzo threefold $X_{3}$ with terminal
singularities admits a vector group variety structure if and only
if its base extension to $\bar{k}$ has precisely three ordinary double
points. In this case, the vector group structure is unique up to isomorphism. 

Furthermore, isomorphism classes of such threefolds are in one-to-one
correspondence with $\mathrm{PGL}_{2}(k)$-orbits of smooth $0$-dimensional
sub-schemes of $\mathbb{P}_{k}^{1}$ of length three. 
\end{thm}

The article is organized as follows. In the first section we collect
standard facts on Grassmannians and basic properties of vector groups
and their actions on varieties. Section two is devoted to a review
of certain classes of linear Schubert sub-varieties of the Grassmannian
$\mathrm{G}(2,5)$ and its smooth linear sections, and of their associated
rational linear projections. These preliminary results are then applied
in the third and fourth sections to derive the proofs of Theorem \ref{thm:MainTh1}
and Theorem \ref{thm:MainTh2}, respectively. \\

\noindent \textbf{Acknowledgements.}  This research was initiated
during the stay of the first and the third authors at Saitama University
in March 2020 at the occasion of the last in-person workshop ``Affine
and Birational Geometry'' held before the Covid-19 pandemic. 

The first author was partially supported by the French ANR project
\textquotedbl FIBALGA\textquotedbl{} ANR-18-CE40-0003 and acknowledge
the support of the EIPHI Graduate School ANR-17-EURE-0002 to the Institute
of Mathematics of Burgundy. The second author was partially funded
by JSPS KAKENHI Grant Number 19K03395. The third author was partially
supported by Fondecyt ANID projects 11190323 and 1200502.

\section{Preliminaries}

We work over a field $k$ of characteristic zero, with a fixed algebraic
closure $\bar{k}$ and associated Galois group $\Gamma=\mathrm{Gal}(\bar{k}/k)$.
We consider $\Gamma$ as the profinite group ${\displaystyle \lim_{\leftarrow}\mathrm{Gal}(k'/k)}$
endowed with the profinite topology, the limit being taken over the
directed set of finite Galois extension $k\subset k'$ of $k$, each
group $\mathrm{Gal}(k'/k)$ being endowed with the discrete topology.
A $k$-variety is an integral $k$-scheme of finite type. 

\subsection{Notation and conventions}

For vector bundles and projective bundles, we follow \cite{EGAII}.
Namely, given a quasi-coherent sheaf $\mathcal{F}$ on a scheme $X$,
we let $\mathrm{Sym}\mathcal{F}$ be the symmetric algebra of $\mathcal{F}$
and we denote by $p:\mathbb{V}_{X}(\mathcal{F})=\mathrm{Spec}_{X}(\mathrm{Sym}^{\cdot}\mathcal{F})\to X$
the ``vector bundle'' over $X$ associated to $\mathcal{F}$ and
by 
\[
\pi:\mathbb{P}_{X}(\mathcal{F})=\mathrm{Proj}_{X}(\mathrm{Sym}\mathcal{F})\cong(\mathbb{V}_{X}(\mathcal{F})\setminus0_{X})/\mathbb{G}_{m,X}\to X,
\]
where $0_{X}$ is the zero section of $p$, its associated ``projective
bundle''. We denote by $\mathcal{O}_{\mathbb{P}_{X}(\mathcal{F})}(1)$
the canonical coherent invertible sheaf of $\mathcal{O}_{\mathbb{P}_{X}(\mathcal{F})}$-modules
associated to $\mathrm{Sym}\mathcal{F}$ viewed as graded sheaf of
modules over itself with the degree shifted by $1$. When $\mathcal{E}$
is a coherent locally free sheaf, $\mathbb{V}_{X}(\mathcal{E})$ is
a usual Zariski locally trivial vector bundle of finite rank and $\mathbb{P}_{X}(\mathcal{E})$
is its associated projective bundle of lines. 

Given a quasi-coherent sheaf $\mathcal{F}$ on a scheme $X$ and an
integer $d\geq0$, we denote by $\rho:\mathbb{G}_{X}(\mathcal{F},d)\to X$
the Grassmann bundle whose $T$-points, where $f:T\to X$ is any $X$-scheme,
are equivalence classes of coherent locally free quotients $\mathcal{E}$
of $f^{*}\mathcal{F}$ of constant rank $d$, two such quotients being
called equivalent if the corresponding surjections $q:f^{*}\mathcal{F}\to\mathcal{E}$
and $q':f^{*}\mathcal{F}\to\mathcal{E}'$ have the same kernel, see
e.g. \cite{Klei69}. We denote by $\rho^{*}\mathcal{F}\to\mathcal{Q}$
the universal coherent locally free quotient of constant rank $d$
and refer the kernel $\mathcal{S}$ of this surjection to as the universal
subsheaf of $\rho^{*}\mathcal{F}$ so that we have the universal exact
sequence $0\to\mathcal{S}\to\rho^{*}\mathcal{F}\to\mathcal{Q}\to0$.
For a coherent locally free sheaf $\mathcal{F}$ of constant rank
$n\geq d+1$, there is a canonical isomorphism $\mathbb{G}_{X}(\mathcal{F},d)\to\mathbb{G}_{X}(\mathcal{F}^{\vee},n-d)$
given on $T$-points by mapping a quotient $q:f^{*}\mathcal{F}\to\mathcal{E}$
with kernel $\mathcal{K}$ to the quotient $f^{*}\mathcal{F}^{\vee}\to\mathcal{K}^{\vee}$.
For $d=1$, $\mathbb{G}_{X}(\mathcal{F},d)\cong\mathbb{P}_{X}(\mathcal{F})$
and the universal exact sequence coincides under this isomorphism
with the relative Euler exact sequence of $\mathbb{P}_{X}(\mathcal{F})$
\[
0\to\Omega_{\mathbb{P}_{X}(\mathcal{F})}^{1}(1):=\Omega_{\mathbb{P}_{X}(\mathcal{F})}^{1}\otimes\mathcal{O}_{\mathbb{P}_{X}(\mathcal{F})}(1)\to\pi^{*}\mathcal{F}\to\mathcal{O}_{\mathbb{P}_{X}(\mathcal{F})}(1)\to0.
\]

\subsection{Grassmannians\label{subsec:Grass-Prelim}}

We summarize basic properties of Grassmannian varieties, see e.g.
\cite{GV15} and \cite[Chapters 3 and 4]{Wey03} for the details.
For a $k$-vector space $V$ of dimension $n\geq2$ and an integer
$1\leq d\leq n-1$, the Grassmann bundle $\rho:\mathbb{G}_{k}(V^{\vee},d)\to\mathrm{Spec}(k)$
is the $d(n-d)$-dimensional Grassmannian whose geometric points correspond
to equivalence classes of $d$-dimensional quotients of $V_{\bar{k}}^{\vee}$,
equivalently to $d$-dimensional $\bar{k}$-vector subspaces $E$
of $V_{\bar{k}}$. 

\subsubsection{Tautological sheaves }

We put $V_{\mathbb{G}_{k}(V^{\vee},d)}^{\vee}=V^{\vee}\otimes_{k}\mathcal{O}_{\mathbb{G}_{k}(V^{\vee},d)}$
and we write 
\[
0\to\mathcal{S}=\mathcal{S}_{\mathbb{G}_{k}(V^{\vee},d)}\to V_{\mathbb{G}_{k}(V^{\vee},d)}^{\vee}\to\mathcal{Q}=\mathcal{Q}_{\mathbb{G}_{k}(V^{\vee},d)}\to0
\]
for the universal sequence of coherent locally free sheaves on $\mathbb{G}_{k}(V^{\vee},d)$.
The sheaf of Kähler differential $\Omega_{\mathbb{G}_{k}(V^{\vee},d)/k}^{1}$
is canonically isomorphic to $\mathcal{H}\mathrm{om}(\mathcal{Q},\mathcal{S})\cong\mathcal{S}\otimes\mathcal{Q}^{\vee}$
and its determinant $\omega_{\mathbb{G}_{k}(V^{\vee},d)}$ is canonically
isomorphic to $(\det\mathcal{S})^{\otimes n-d}\otimes(\det\mathcal{Q}^{\vee})^{\otimes d}\cong(\det\mathcal{Q}^{\vee})^{\otimes n}$.
The $k$-vector spaces $H^{0}(\mathbb{G}_{k}(V^{\vee},d),\mathcal{Q})$
and $H^{0}(\mathbb{G}_{k}(V^{\vee},d),\mathcal{S}^{\vee})$ are canonically
isomorphic to $V^{\vee}$ and $V$ respectively, $H^{1}(\mathbb{G}_{k}(V^{\vee},d),\Omega_{\mathbb{G}_{k}(V^{\vee},d)/k}^{1})\cong k$
and all other cohomology spaces of $\mathcal{Q}$, $\mathcal{Q}^{\vee}$,\emph{
$\mathcal{S}$, $\mathcal{S}^{\vee}$ }and $\Omega_{\mathbb{G}_{k}(V^{\vee},d)/k}^{1}$
are zero.

\subsubsection{Pl\"ucker embedding and automorphisms\label{subsec:Plucker-Aut-Iso}}

We denote by $j_{P}:\mathbb{G}_{k}(V^{\vee},d)\to\mathbb{P}_{k}(\Lambda^{d}V^{\vee})$
the \emph{Pl\"ucker embedding}, that is,\emph{ }the closed immersion
determined by the surjection $\Lambda^{d}V_{\mathbb{G}_{k}(V^{\vee},d)}^{\vee}\to\det\mathcal{Q}$
induced by the universal quotient homomorphism. Letting $\mathrm{Aut}_{k}(\mathbb{P}_{k}(\Lambda^{d}V^{\vee}),\mathbb{G}_{k}(V^{\vee},d))$
be the stabilizer of $j_{P}(\mathbb{G}_{k}(V^{\vee},d))$ in $\mathrm{Aut}_{k}(\mathbb{P}_{k}(\Lambda^{d}V^{\vee}))$,
it follows from \cite{Chow49} that the composition of the homomorphism
of $k$-group schemes 
\begin{equation}
\mathrm{PGL}_{k}(V^{\vee})=\mathrm{Aut}_{k}(\mathbb{P}_{k}(V^{\vee}))\to\mathrm{Aut}{}_{k}(\mathbb{P}_{k}(\Lambda^{d}V^{\vee}),\mathbb{G}_{k}(V^{\vee},d)),\,\varphi\mapsto\Lambda^{d}\varphi\label{eq:Aut-Plucker}
\end{equation}
with the restriction homomorphism $\mathrm{Aut}_{k}(\mathbb{P}_{k}(\Lambda^{d}V^{\vee}),\mathbb{G}_{k}(V^{\vee},d))\to\mathrm{Aut}_{k}(\mathbb{G}_{k}(V^{\vee},d))$
is a closed immersion, which is an isomorphism when $n\neq2d$ (otherwise,
its image is a $k$-subgroup scheme of index $2$). 

\subsection{Vector groups and vector group varieties\vspace{-0.5em}}

\subsubsection{Vector groups }

A \emph{vector $k$-group} is an abelian unipotent algebraic $k$-group
scheme. By \cite[IV.2.4]{DeGa70}, there is an equivalence between
the category of finite dimensional $k$-vector spaces and the category
of vector $k$-groups, given by the map associating to a finite dimensional
$k$-vector space $V$ the $k$-group scheme $(\mathbb{V}_{k}(V^{\vee}),+)$,
where the comorphism of the $k$-group scheme structure $+$ is induced
by the diagonal homomorphism $V^{\vee}\to V^{\vee}\oplus V^{\vee}$,
and to a $k$-linear homomorphism $f:W\to V$ the $k$-group homomorphism
$(\mathbb{V}_{k}(W^{\vee}),+)\to(\mathbb{V}_{k}(V^{\vee}),+)$ induced
by $^{t}f:V^{\vee}\to W^{\vee}$. The choice of a $k$-basis of $V$
determines an isomorphism of $k$-groups schemes $(\mathbb{V}_{k}(V^{\vee}),+)\cong\mathbb{G}_{a,k}^{n}$.
We will repeatedly use the following elementary results.
\begin{lem}
\label{lem:Unipotent-basic-prop}With the notation above, the following
hold:

\begin{enumerate}

\item Every $k$-subgroup and quotient $k$-group of a vector $k$-group
is a vector $k$-group.

\item Every extension $0\to\mathbb{U}'\to\mathbb{U}\to\mathbb{U}''\to0$
of vector $k$-groups has a splitting $h:\mathbb{U}\stackrel{\cong}{\to}\mathbb{U}'\times\mathbb{U}''$. 

\item The set $M^{\mathbb{U}}$ of $\mathbb{U}$-invariants of a
rational $\mathbb{U}$-module $M$ of finite positive dimension is
nonzero. 

\end{enumerate}
\end{lem}

Recall \cite[I.3]{GIT} (see also \cite[4.2]{HuyLe}) that a \emph{quasi-coherent
$G$-sheaf of $\mathcal{O}_{X}$-modules} on a $k$-scheme $X$ with
an action $\mu:G\times X\to X$ of an algebraic $k$-group $G$ is
a pair $(\mathcal{F},\theta)$ consisting of a coherent sheaf of $\mathcal{O}_{X}$-modules
$\mathcal{F}$ and an isomorphism $\theta:\mu^{*}\mathcal{F}\stackrel{\cong}{\rightarrow}\mathrm{p}_{2}^{*}\mathcal{F}$
of coherent sheaves of $\mathcal{O}_{G\times X}$-modules, called
a \emph{$G$-linearization} of $\mathcal{F}$, that satisfies the
cocycle relation $(m_{G}\times\mathrm{id}_{X})^{*}\theta=\mathrm{p}_{23}^{*}\theta\circ(\mathrm{id}_{G}\times\mu)^{*}\theta$
on $G\times G\times X$, where $m_{G}:G\times G\to G$ is the group
law on $G$ and where $\mathrm{p}_{23}:G\times G\times X\rightarrow G\times X$
is the projection onto the last two factors. In particular, when $G$
acts trivially on $X$, a $G$-linearization of $\mathcal{F}$ is
the same as a homomorphism of $G$ into the group $\mathrm{Aut}_{X}(\mathcal{F})$
of $\mathcal{O}_{X}$-module automorphisms of $\mathcal{F}$. Two
$G$-linearizations $\theta$ and $\theta'$ of $\mathcal{F}$ are
called \emph{equivalent} if there exists an $\mathcal{O}_{X}$-module
automorphism $\varphi$ of $\mathcal{F}$ such that $\mathrm{p}_{2}^{*}\varphi\circ\theta'=\theta\circ\mu^{*}\varphi$. 
\begin{lem}
\label{lem:linearization} Let $X$ be a normal $k$-variety endowed
with an action of a vector $k$-group $\mathbb{U}$. Then every coherent
invertible sheaf of $\mathcal{O}_{X}$-modules $\mathcal{L}$ admits
a $\mathbb{U}$-linearization $\theta_{\mathcal{L}}$ unique up to
equivalence. 
\end{lem}

\begin{proof}
This follows from \cite[Lemma 2.13]{Bri15} and the fact that since
$\mathbb{U}\cong\mathbb{A}_{k}^{n}$, for every normal $k$-variety
$X$, the pullback homomorphisms $\mathrm{p}_{X}^{*}:H^{i}(X,\mathcal{O}_{X}^{*})\to H^{i}(X\times\mathbb{U},\mathcal{O}_{X\times\mathbb{U}}^{*})$,
$i=0,1$, are isomorphisms. 
\end{proof}

\subsubsection{Vector group structures and vector group varieties}
\begin{defn}
A \emph{vector group variety} is a $k$-variety $X$ endowed with
an effective action $\mu:\mathbb{U}\times X\to X$ of a vector $k$-group
$\mathbb{U}$ which has a Zariski dense open orbit $U_{X}$. The action
$\mu$ is said to define a \emph{vector group} \emph{structure} on
$X$. Two vector group structures $\mu:\mathbb{U}\times X\to X$ and
$\mu':\mathbb{U}'\times X\to X$ on $X$ are said to have the same
equivalence class if there exists an isomorphism of $k$-groups $\alpha:\mathbb{U}'\to\mathbb{U}$
and a $k$-automorphism $\varphi$ of $X$ such that $\varphi\circ\mu'=\mu\circ(\alpha\times\varphi)$. 
\end{defn}

\begin{lem}
\label{lem:open-orbit} Let $X$ be a vector group variety with open
orbit $U_{X}$. Then $U_{X}$ is a trivial $\mathbb{U}$-torsor. In
particular, $U_{X}$ contains a $k$-point of $X$. 
\end{lem}

\begin{proof}
Since the action of $\mathbb{U}$ is effective and $U_{X}$ is Zariski
dense, the morphism $(\mu,\mathrm{p}_{2}):\mathbb{U}\times_{k}U_{X}\to U_{X}\times U_{X}$
is an isomorphism, i.e., $U_{X}$ endowed with the induced action
of $\mathbb{U}$ is a $\mathbb{U}$-torsor. The conclusion then follows
from the additive form of Hilbert Theorem $90$ which asserts that
every such torsor is trivial. 
\end{proof}
\begin{example}
Since all orbits of unipotent group actions on a quasi-affine $k$-variety
are closed \cite[Theorem 2]{Ro61}, Lemma \ref{lem:open-orbit} implies
that a quasi-affine vector group variety $X$ is a trivial $\mathbb{U}$-torsor.
In particular, $\mathbb{A}_{k}^{n}$ is the unique quasi-affine $k$-variety
with a $\mathbb{G}_{a,k}^{n}$-structure and this structure is unique
up to isomorphism. On the other hand, by Sumihiro's equivariant completion
\cite[Theorem 3]{Su74}, every normal vector group $k$-variety $X$
admits a $\mathbb{U}$-equivariant open immersion $j:X\hookrightarrow\bar{X}$
into a complete vector group $k$-variety $\bar{X}$. When $X$ is
smooth, the existence of equivariant resolution of singularities \cite[Theorem 3.36 and Proposition 3.9.1]{Ko07}
implies in addition that $\bar{X}$ can be chosen to be smooth and
such that $\bar{X}\setminus j(X)$ is the support of a $\mathbb{U}$-stable
smooth normal crossing divisor on $\bar{X}$. 
\end{example}

\begin{prop}
\label{prop:adrien} Let $X$ be a $k$-variety endowed with a vector
group structure $\mu:\mathbb{U}\times X\to X$, let $f:X\to Y$ be
a proper morphism to a $k$-variety $Y$ such that $f_{*}\mathcal{O}_{X}=\mathcal{O}_{Y}$
and let $i:F\hookrightarrow X$ be the scheme-theoretic fiber of $f$
over a $k$-point $y_{0}$ of $Y$ in the image by $f$ of the open
$\mathbb{U}$-orbit $U_{X}$. Then there exists an extension of vector
$k$-groups $0\to\mathbb{U}'\stackrel{a}{\to}\mathbb{U}\stackrel{b}{\to}\bar{\mathbb{U}}\to0$
such that the following hold: 

\begin{enumerate}

\item The variety $Y$ is endowed with a vector group structure $\mu_{Y}:\bar{\mathbb{U}}\times_{k}Y\to Y$
with open $\bar{\mathbb{U}}$-orbit $U_{Y}$ and $f:X\to Y$ is equivariant
with respect to the homomorphism $b:\mathbb{U}\to\bar{\mathbb{U}}$. 

\item The scheme $F$ is a $k$-variety endowed with a vector group
structure $\mu_{F}:\mathbb{U}'\times_{k}F\to F$ and the closed immersion
$i:F\hookrightarrow X$ is equivariant with respect to the homomorphism
$a:\mathbb{U}'\to\mathbb{U}$. 

\item Given any section $c:\bar{\mathbb{U}}\to\mathbb{U}$ of $b:\mathbb{U}\to\bar{\mathbb{U}}$,
the morphism \[j:\mu\circ(c\times i)\circ(\mu_{Y}^{-1}(\cdot,y_{0})\times\mathrm{id}_{F}):U_{Y}\times F\stackrel{\cong}{\to}\bar{\mathbb{U}}\times F\to\mathbb{U}\times X\to X\]
is a $\mathbb{U}'\times\bar{\mathbb{U}}$-equivariant open immersion
with image $f^{-1}(U_{Y})$. 

\end{enumerate}
\end{prop}

\begin{proof}
By Blanchard's lemma \cite[Theorem 7.2.1]{Bri17}, there exists a
unique action $\nu:\mathbb{U}\times Y\to Y$ such that $f$ is $\mathbb{U}$-equivariant.
Let $\mathbb{U}'\subset\mathbb{U}$ be the stabilizer of $y_{0}$
and let $\bar{\mathbb{U}}=\mathbb{U}/\mathbb{U}'$. Since $y_{0}$
belongs to $f(U_{X})$, the $\mathbb{U}$-orbit of $y_{0}$ is a constructible
set which is not contained in any proper closed subset of $Y$. It
thus contains a Zariski dense open subset of $Y$, hence, being homogeneous
under the action of $\mathbb{U}$, is a Zariski dense open subset
of $Y$. This implies that $\mathbb{U}'$ acts trivially on $Y$ and
that the induced action $\mu_{Y}:\bar{\mathbb{U}}\times Y\to Y$ of
$\bar{\mathbb{U}}$ is a vector group structure on $Y$ with the property
that $f:X\to Y$ is equivariant with respect to the $k$-group homomorphism
$b:\mathbb{U}\to\bar{\mathbb{U}}$. This proves (1). The closed subscheme
$F$ is $\mathbb{U}'$-stable, with $\mathbb{U}'$-action $\mu_{F}:\mathbb{U}'\times F\to F$
induced by $\mu$. Given a section $c:\bar{\mathbb{U}}\to\mathbb{U}$
of $b$, we have the following cartesian square of $\bar{\mathbb{U}}$-equivariant
morphisms \[\xymatrix@C=5em{\bar{\mathbb{U}}\times F\ar[d]_{\mathrm{id}_{\bar{\mathbb{U}}}\times f}\ar[r]^{\mu\circ( c \times i)} & X\ar[d]^{f}\\ \bar{\mathbb{U}}\times\{y_0\}\ar[r]^{\mu_{Y}(\cdot,y_{0})} & Y } \]where
$\bar{\mathbb{U}}$ acts on $\bar{\mathbb{U}}\times F$ and $\bar{\mathbb{U}}\times\{y_{0}\}$
by translations on the first factor and on $X$ by $\mu\circ(c\times\mathrm{id}_{X})$.
For every $v\in\bar{\mathbb{U}}(k)$, $\mu\circ(c(v)\times\mathrm{id}_{X})$
is an automorphism of $X$ which maps $F$ isomorphically onto the
scheme-theoretic fiber of $f$ over the point $\mu_{Y}(v,y_{0})$.
Since $\mu_{Y}(\cdot,y_{0}):\bar{\mathbb{U}}\times\{y_{0}\}\to Y$
is an open immersion with image $U_{Y}$, it follows that $\mu\circ(c\times i)$
is an open immersion with image $f^{-1}(U_{Y})$. Furthermore, $\mu\circ(c\times i)$
is equivariant for the isomorphism $h=(c,a):\bar{\mathbb{U}}\times\mathbb{U}'\to\mathbb{U}$
with respect to the product action of $\bar{\mathbb{U}}\times\mathbb{U}'$
on $\bar{\mathbb{U}}\times F$ and the $\mathbb{U}$-action on $X$.
Assertion (3) follows. Finally, assertion (2) follows from the observation
that the intersection of the inverse image of $U_{X}\subset f^{-1}(U_{Y})$
by $\mu\circ(c\times i)$ with $\{0\}\times F$ is a Zariski dense
$\mathbb{U}'$-stable open subset $U_{F}$ of $F$ which is a principal
homogeneous space of $\mathbb{U}'$. 
\end{proof}
Recall that a coherent locally free sheaf $\mathcal{E}$ on a $k$-variety
$X$ is called \emph{simple} if its only endomorphisms are scalar
homothethies. Lemma \ref{lem:linearization} and Proposition \ref{prop:adrien}
imply the following result. 
\begin{cor}
\label{cor:Simple-no-struct} For a simple coherent locally free sheaf
$\mathcal{E}$ of rank $r\geq2$ on a normal $k$-variety $X$, the
total space of the projective bundle $\pi:\mathbb{P}_{X}(\mathcal{E})\to X$
does not admit a vector group structure. 
\end{cor}

\begin{proof}
Assume that $\mathbb{P}_{X}(\mathcal{E})$ admits a vector group structure
given by the action of vector group $\mathbb{U}$. By Lemma \ref{lem:linearization},
the invertible sheaf $\mathcal{O}_{\mathbb{P}_{X}(\mathcal{E})}(1)$
is canonically $\mathbb{U}$-linearized. By Proposition \ref{prop:adrien},
$\pi:\mathbb{P}_{X}(\mathcal{E})\to X$ is $\mathbb{U}$-equivariant
for a uniquely determined $\mathbb{U}$-action on $X$, which factors
through an effective action of a non-trivial quotient $\bar{\mathbb{U}}=\mathbb{U}/\mathbb{U}'$
defining a vector group structure on $X$. Since $\pi$ is $\mathbb{U}$-equivariant,
$\mathcal{E}\cong\pi_{*}\mathcal{O}_{\mathbb{P}_{X}(\mathcal{E})}(1)$
is endowed with an induced $\mathbb{U}$-linearization for the action
on $X$, hence with a linearization for the trivial action of the
positive dimensional vector group $\mathbb{U}'$. The latter is determined
by some group homomorphism $\mathbb{U}'\to\mathrm{Aut}_{X}(\mathcal{E})$
which is injective by Proposition \ref{prop:adrien}. But this is
impossible since $\mathrm{Aut}_{X}(\mathcal{E})\cong\mathbb{G}_{m,X}$
by hypothesis. 
\end{proof}
\vspace{-1.6em}

\section{Linear sections of $G(2,5)$ and their rational linear projections }

We collect results concerning linear Schubert sub-varieties of dimension
$\geq2$ of the Grassmannian $\mathrm{G}(2,5)$ of $2$-dimensional
$k$-vector subspaces of $k^{\oplus5}$ and of its smooth linear sections
in the Pl\"ucker embedding. We then review the description of the
rational maps given by projections with respect to these linear Schubert
sub-varieties. Over algebraically closed fields, all this material
is classical, see e.g. \cite{Don77,Fuj81,PV99,Todd30}. \vspace{-0.7em}

\subsection{Linear Schubert subvarieties of $G(2,5)$ and its hyperplane sections}

We put $G=\mathbb{G}_{k}(V^{\vee},2)\cong G(2,5)$ for some fixed
$5$-dimensional $k$-vector space $V$. We denote by $0\to\mathcal{S}\to V_{G}^{\vee}\to\mathcal{Q\to}0$
the universal sequence on $G$ and by $j_{P}:G\hookrightarrow\mathbb{P}_{k}(\Lambda^{2}V^{\vee})$
the Pl\"ucker embedding. For any algebraic extension $k'$ of $k$,
we interpret $k'$-points of $G_{k'}$ either as $2$-dimensional
$k'$-vector subspaces $E\subset V_{k'}$ or as their corresponding
lines $\mathbb{P}_{k'}(E^{\vee})$ in $\mathbb{P}_{k'}(V_{k'}^{\vee})$.
For concrete examples, we fix the following coordinate convention: 
\begin{notation}
\label{nota:Plucker-Coord}For a chosen basis $e_{1},\ldots,e_{5}$
of $V$ with dual basis $e_{1}^{\vee},\ldots,e_{5}^{\vee}$, we identify
$\mathbb{P}_{k}(V^{\vee})$ with $\mathbb{P}_{k}^{4}$ and $\mathrm{G}(2,5)$
with the closed subvariety of $\mathbb{P}_{k}(\Lambda^{2}V^{\vee})=\mathbb{P}_{k}^{9}$
endowed with the Pl\"ucker coordinates $\ensuremath{w_{ij}=e_{i}^{\vee}\wedge e_{j}^{\vee}}$,
$\ensuremath{1\leq i<j\leq5}$, defined by the equations  \[ \left\{\begin{array}{l} w_{12}w_{34}-w_{13}w_{24}+w_{14}w_{23}=0\\ w_{12}w_{35}-w_{13}w_{25}+w_{15}w_{23}=0\\ w_{12}w_{45}-w_{14}w_{25}+w_{15}w_{24}=0\\ w_{13}w_{45}-w_{14}w_{35}+w_{15}w_{34}=0\\ w_{23}w_{45}-w_{24}w_{35}+w_{25}w_{34}=0 \end{array} \right.\] 
\end{notation}

\subsubsection{Solids and planes in $\mathrm{G}(2,5)$ }

We consider the following linear Schubert sub-varieties of $G$:
\begin{defn}
Let $\{V_{1}\subset V_{3}\subset V_{4}\}$ be a partial flag of $k$-vector
subspaces of $V$, with $\dim_{k}V_{i}=i$:

$\bullet$ The\emph{ $\sigma_{3,0}$-solid} $\sigma_{3,0}(V_{1})\cong\mathbb{P}_{k}((V/V_{1})^{\vee})$
associated to $V_{1}$ is the zero scheme of the homomorphism $V_{1,G}\hookrightarrow V_{G}\to\mathcal{S}^{\vee}$.
Its intersection $\sigma_{3,1}(V_{1}\subset V_{4})\cong\mathbb{P}_{k}((V_{4}/V_{1})^{\vee})$
with the zero scheme of the homomorphism $(V/V_{4})_{G}^{\vee}\hookrightarrow V_{G}^{\vee}\to\mathcal{Q}$
is called the\emph{ $\sigma_{3,1}$-plane} associated to $\{V_{1}\subset V_{4}\}$. 

$\bullet$ The\emph{ $\sigma_{2,2}$-plane} $\sigma_{2,2}(V_{3})=\mathbb{G}_{k}(V_{3}^{\vee},2)\cong\mathbb{P}_{k}(V_{3})$
associated to $V_{3}$ is the zero scheme of the homomorphism $(V/V_{3})_{G}^{\vee}\hookrightarrow V_{G}^{\vee}\to\mathcal{Q}$. 
\end{defn}

The above sub-schemes are linear subspaces of $G$ in the Pl\"ucker
embedding $G\subset\mathbb{P}_{k}(\Lambda^{2}V^{\vee})$, given respectively
as the intersections $G\cap\mathbb{P}_{k}(\Lambda^{2}V^{\vee}/\Lambda^{2}(V/V_{1})^{\vee})$,
$\sigma_{3,0}(V_{1})\cap\mathbb{P}_{k}(\Lambda^{2}V_{4}^{\vee})$
and $G\cap\mathbb{P}_{k}(\Lambda^{2}V_{3}^{\vee})$. It follows from
\cite{Todd30} that they are the only linear $k$-subspaces of dimension
$\geq2$ contained in $G$.\footnote{The two types of planes in $G$
are called planes of the first and second system in \cite{Todd30},
and in \cite{Fuj81}, planes of vertex type and of non-vertex type,
respectively.} Geometrically, the closed points of $\sigma_{3,0}(V_{1})_{\bar{k}}$
and $\sigma_{3,1}(V_{1}\subset V_{4})_{\bar{k}}$ correspond respectively
to lines in $\mathbb{P}_{\bar{k}}(V_{\bar{k}}^{\vee})$ passing through
the point $\mathbb{P}_{\bar{k}}(V_{1,\bar{k}}^{\vee})$ and its subset
of those which are contained in the subspace $\mathbb{P}_{k}(V_{4}^{\vee})$.
The closed points of $\sigma_{2,2}(V_{2})_{\bar{k}}$ correspond to
lines contained in the subspace $\mathbb{P}_{\bar{k}}(V_{3,\bar{k}}^{\vee})$
of $\mathbb{P}_{\bar{k}}(V_{\bar{k}}^{\vee})$. 
\begin{rem}
The conormal sheaves in $G$ of a solid $\Pi=\sigma_{3,0}(V_{1})$
and of a plane $\Xi=\sigma_{2,2}(V_{3})$ are canonically isomorphic
to $\Omega_{\Pi}^{1}(1)\otimes V_{1,\Pi}$ and to $\Omega_{\Xi}^{1}(1)\otimes(V/V_{3})_{\Xi}^{\vee}$,
respectively. Indeed, since $\Pi$ is the zero scheme of $V_{1,G}\to\mathcal{S}^{\vee}$,
its conormal sheaf $\mathcal{C}_{\Pi/G}$ is isomorphic to $(\mathcal{S}\otimes V_{1,G})|_{\Pi}\cong\mathcal{S}|_{\Pi}\otimes V_{1,\Pi}$
and the canonical isomorphism $\mathcal{S}|_{\Pi}\cong\Omega_{\Pi}^{1}(1)$
is given by the following commutative diagram \[\xymatrix@R1em@C1em{ &  &  0 \ar[d] &  0 \ar[d] &  \\ 0  \ar[r] & \Omega_{\Pi}^{1}(1) \ar[r] \ar@{=}[d] & (V/V_{1})_{\Pi}^{\vee} \ar[r] \ar[d]  & \mathcal{O}_{\Pi}(1) \ar[r] \ar[d] & 0 \\ 0  \ar[r] & \mathcal{S}|_{\Pi} \ar[r] \ar[dr]^{0} & V_{\Pi}^{\vee} \ar[r] \ar[d] & \mathcal{Q}|_{\Pi} \ar[r] \ar[d] & 0 \\ 
 & & V_{1,\Pi}^{\vee} \ar@{=}[r] \ar[d] & V_{1,\Pi}^{\vee} \ar[d] \\ & & 0  & 0} \] Similarly, since $\Xi$ is a codimension $4$ local complete intersection
equal to the zero-scheme of the homomorphism $(V/V_{3})_{G}^{\vee}\to\mathcal{Q}$,
the conormal sheaf $\mathcal{C}_{\Xi/G}$ is canonically isomorphic
to $(\mathcal{Q}^{\vee}\otimes(V/V_{3})_{G}^{\vee})|_{\Xi}\cong(\mathcal{Q}^{\vee}|_{\Xi})\otimes(V/V_{3})_{\Xi}^{\vee}$,
and since $\mathcal{Q}|_{\Xi}$ equals the universal quotient sheaf
on $\Xi\cong\mathbb{G}_{k}(V_{3}^{\vee},2)$, the canonical isomorphism
$\mathbb{G}_{k}(V_{3}^{\vee},2)\cong\mathbb{P}_{k}(V_{3})$ gives
the identification $\mathcal{Q}^{\vee}|_{\Xi}\cong\Omega_{\Xi}^{1}(1)$.
\end{rem}

\begin{notation}
\label{nota:Stabilizers-Schubert}Given a $d$-dimensional $k$-vector
subspace $V_{d}\subset V$, $1\leq d\leq4$, we denote by $\Delta_{V_{d}}$
the stabilizer in $\mathrm{GL}_{k}(V^{\vee})$ of the subspace $(V/V_{d})^{\vee}\subset V^{\vee}$.
The choice of a splitting $V\cong V_{d}\oplus(V/V_{d})$ identifies
this $k$-group scheme with that consisting of block-matrices of the
form 
\begin{equation}
M(A_{5-d},A_{d},U)=\left(\begin{array}{cc}
A_{5-d} & U\\
0 & A_{d}
\end{array}\right)\in\mathrm{GL}_{k}((V/V_{d})^{\vee}\oplus V_{d}^{\vee})\label{eq:Stabilizer-solid}
\end{equation}
with $A_{5-d}\in\mathrm{GL}_{k}((V/V_{d})^{\vee})$, $A_{d}\in\mathrm{GL}_{k}(V_{d}^{\vee})$
and $U\in\mathrm{Hom}_{k}(V_{d}^{\vee},(V/V_{d})^{\vee})$. The associated
subgroups $\mathrm{P}\Delta_{V_{d}}$ of $\mathrm{PGL}_{k}(V^{\vee})$
correspond under the isomorphism $\mathrm{PGL}_{k}(V^{\vee})\cong\mathrm{Aut}_{k}(G)$
of (\ref{eq:Aut-Plucker}) to the stabilizer subgroups of the solid
$\sigma_{3,0}(V_{1})$ if $d=1$ and of the plane $\sigma_{2,2}(V_{3})$
if $d=3$. 
\end{notation}

\subsubsection{\label{subsec:Hyperplane-Sec} Smooth hyperplane sections }

By a hyperplane section of $G$, we mean the zero scheme $Z_{\langle s\rangle}$
of a non-zero global section $s\in H^{0}(G,\Lambda^{2}\mathcal{Q})=\Lambda^{2}V^{\vee}$,
equivalently the intersection of $G$ in the Pl\"ucker embedding
with the hyperplane $\mathbb{P}_{k}(\Lambda^{2}V^{\vee}/\langle s\rangle)$
of $\mathbb{P}_{k}(\Lambda^{2}V^{\vee})$. Denote by $\tilde{s}:V\to V^{\vee}$
the $k$-linear homomorphism corresponding to the form $s$ under
the canonical isomorphism $\mathrm{Hom}_{k}(V\otimes V,k)\cong\mathrm{Hom}_{k}(V,V^{\vee})$.
The fivefold $Z_{\langle s\rangle}$ is the \emph{isotropic Grassmannian}
$\mathrm{I}_{s}\mathbb{G}_{k}(V^{\vee},2)$: a closed point $E\subset V_{\bar{k}}$
of $G_{\bar{k}}$ belongs to $(Z_{\langle s\rangle})_{\bar{k}}$ if
and only if the homomorphism $\tilde{s}_{\bar{k}}|_{E}:E\to V_{\bar{k}}\to V_{\bar{k}}^{\vee}$
has image contained in $(V_{\bar{k}}/E)^{\vee}$, hence if and only
if $E$ is $s_{\bar{k}}$-isotropic. Considering the conormal sequence
\[
0\to\mathcal{C}_{Z_{\langle s\rangle}/G}=\Lambda^{2}\mathcal{Q}^{\vee}|_{Z_{\langle s\rangle}}\otimes_{\mathcal{O}_{Z\langle s\rangle}}\langle s\rangle_{Z_{\langle s\rangle}}\stackrel{d}{\to}\Omega_{G/k}^{1}|_{Z_{\langle s\rangle}}=\mathcal{H}om(\mathcal{Q}|_{Z_{\langle s\rangle}},\mathcal{S}|_{Z_{\langle s\rangle}})\to\Omega_{Z_{\langle s\rangle}/k}^{1}\to0
\]
we see that $Z_{\langle s\rangle,\bar{k}}$ is smooth at $E\subset V_{\bar{k}}$
if and only if the map 
\[
d|_{E}:\Lambda^{2}E\subset\mathrm{Hom}_{k}(E^{\vee}\otimes E^{\vee},k)\cong\mathrm{Hom}_{k}(E^{\vee},E)\to\mathrm{Hom}_{\bar{k}}(E^{\vee},(V_{\bar{k}}/E)^{\vee}),\quad f\mapsto\tilde{s}_{\bar{k}}|_{E}\circ f
\]
 is nonzero, hence if and only if $E\not\subset\mathrm{Ker}\tilde{s}_{\bar{k}}$.
In particular, $Z_{\langle s\rangle}$ is smooth if and only if $\tilde{s}$
has rank $4$. This yields a functorial correspondence between $k$-points
$\langle s\rangle^{\vee}$ of $\mathbb{P}_{k}(\Lambda^{2}V)\setminus\mathbb{G}_{k}(V,2)$
and smooth hyperplane sections $Z_{\langle s\rangle}$ of $G$ from
which we get in particular that the action of $\mathrm{Aut}_{k}(G)(k)\cong\mathrm{PGL}_{5}(k)$
on the set of smooth hyperplane sections $Z_{\langle s\rangle}$ of
$G$ is transitive. 

Given a smooth hyperplane section $Z_{\langle s\rangle}$, the $s$-orthogonal
$V^{\bot}=\mathrm{Ker}\tilde{s}$ of $V$ has dimension $1$. We put
$\bar{V}=V/V^{\bot}$ and denote by $\bar{s}\in\Lambda^{2}\bar{V}^{\vee}$
the symplectic form on $\bar{V}$ induced by $s$. Letting $\bar{G}=\mathbb{G}_{k}(\bar{V}^{\vee},2)$
with universal sequence $0\to\mathcal{S}_{\bar{G}}\to\bar{V}_{\bar{G}}^{\vee}\to\mathcal{Q}_{\bar{G}}\to0$,
the zero scheme $Q_{\langle\bar{s}\rangle}\subset\bar{G}$ of $\bar{s}$
is the L\emph{agrangian Grassmannian} \emph{$\mathrm{L}_{\bar{s}}\mathbb{G}_{k}(\bar{V}^{\vee},2)$}
whose $k$-points are the maximal $\bar{s}$-isotropic subspaces of
$\bar{V}$. The Pl\"ucker embedding $\bar{G}\hookrightarrow\mathbb{P}_{k}(\Lambda^{2}\bar{V}^{\vee})$
induces a closed immersion of $Q_{\langle\bar{s}\rangle}$ in $\mathbb{P}_{k}(\Lambda^{2}\bar{V}^{\vee}/\langle\bar{s}\rangle)$
as the zero scheme of the non-degenerate quadratic form $\bar{q}$
associated to the symmetric bi-linear form $\bar{b}\in\mathrm{Sym}^{2}(\Lambda^{2}\bar{V}^{\vee}/\langle\bar{s}\rangle)$
induced by the bi-linear form $b:\Lambda^{2}\bar{V}\otimes\Lambda^{2}\bar{V}\to k$,
$\bar{u}_{1}\wedge\bar{v}_{1}\otimes\bar{u}_{2}\wedge\bar{v}_{2}\mapsto(\bar{s}\wedge\bar{s})(\bar{u}_{1}\wedge\bar{v}_{1}\wedge\bar{u}_{2}\wedge\bar{v}_{2})$. 
\begin{lem}
\label{lem:Solids-and-planes-Z5}With the notation above, the following
hold:

\begin{enumerate}

\item Every smooth hyperplane section $Z_{\langle s\rangle}$ of
$G$ contains a unique solid $\Pi_{\langle s\rangle}=\sigma_{3,0}(V^{\bot})$
and conversely every solid of $G$ is contained in a (non-unique)
smooth hyperplane section of $G$.

\item A plane $\sigma_{2,2}(V_{3})$ of $G$ is contained in a smooth
hyperplane section $Z_{\langle s\rangle}$ if and only if $V^{\bot}\subset V_{3}$
and $V_{3}/V^{\bot}$ is $\bar{s}$-isotropic. In other words, the
$\sigma_{2,2}$-planes in $Z_{\langle s\rangle}$ are in one-to-one
correspondence with the $k$-points of $Q_{\langle\bar{s}\rangle}$
and they intersect the unique solid $\Pi_{\langle s\rangle}$ of $Z_{\langle s\rangle}$
along lines. 

\end{enumerate}
\end{lem}

\begin{proof}
A solid $\sigma_{3,0}(V_{1})$ of $G$ is contained in $Z_{\langle s\rangle}$
if and only if $V=V_{1}^{\bot}$ hence, since by hypothesis $V^{\bot}$
is $1$-dimensional, if and only if $V_{1}=V^{\bot}$. Conversely,
for every $1$-dimensional $k$-vector subspace $V_{1}$ of $V$,
any choice of symplectic form $\bar{s}\in\Lambda^{2}(V/V_{1})^{\vee}$
on the $4$-dimensional $k$-space $V/V_{1}$ determines through the
inclusion $\Lambda^{2}(V/V_{1})^{\vee}\subset\Lambda^{2}V^{\vee}$
a skew-symmetric form $s\in\Lambda^{2}V^{\vee}$ with $\mathrm{Ker}\tilde{s}=V_{1}$,
hence a corresponding smooth hyperplane section $Z_{\langle s\rangle}$
of $G$ containing $\sigma_{3,0}(V_{1})$. A plane $\sigma_{2,2}(V_{3})$
of $G$ is contained in $Z_{\langle s\rangle}$ if and only if every
$2$-dimensional $k$-vector subspace $E\subset V_{3}$ is $s$-isotropic,
which holds if and only if $V_{3}$ is $s$-isotropic. This property
is equivalent to the fact that $V_{3}$ contains $V^{\bot}$ and $\bar{V}_{3}=V_{3}/V^{\bot}$
is an $\bar{s}$-isotropic subspace of $\bar{V}$. Every $\sigma_{2,2}$-plane
with this property intersects $\Pi_{\langle s\rangle}\cong\mathbb{P}_{k}(\bar{V}^{\vee})$
along the line $\mathbb{P}_{k}(\bar{V}_{3}^{\vee})$. 
\end{proof}
\begin{rem}
\label{rem:Normal-sheaf-solid-Z}For a smooth hyperplane section $Z=Z_{\langle s\rangle}$
of $G$, the exact sequence
\[
0\to\mathcal{C}_{Z/G}|_{\Pi}\cong\Lambda^{2}\mathcal{Q}^{\vee}|_{\Pi}\stackrel{\bar{s}}{\to}\mathcal{C}_{\Pi/G}\cong\Omega_{\Pi}^{1}(1)\to\mathcal{C}_{\Pi/Z}\to0
\]
identifies the conormal sheaf $\mathcal{C}_{\Pi/Z}$ of the unique
solid $\Pi=\Pi_{\langle s\rangle}\cong\mathbb{P}_{k}(\bar{V}^{\vee})$
of $Z_{\langle s\rangle}$ with the null-correllation rank $2$ locally
sheaf $\mathcal{N}_{\langle\bar{s}\rangle}$ associated to the symplectic
form $\bar{s}\in\Lambda^{2}\bar{V}^{\vee}$, that is, the sheaf whose
fiber over a closed point $\ell\subset\bar{V}_{\bar{k}}$ of $\mathbb{P}_{\bar{k}}(\bar{V}_{\bar{k}}^{\vee})$
is the quotient $\ell^{\bot}/\ell$, where $\ell^{\bot}$ is the $\bar{s}_{\bar{k}}$-orthogonal
of $\ell$. 
\end{rem}

\subsubsection{\label{subsec:Codim2-Sec}Smooth linear sections of codimension $2$ }

A codimension $2$ linear section of $G$ is the zero scheme $W_{L}$
of a homomorphism $L_{G}\to\Lambda^{2}\mathcal{Q}$ determined by
a $2$-dimensional $k$-vector subspace $L\subset H^{0}(G,\Lambda^{2}\mathcal{Q})=\Lambda^{2}V^{\vee}$,
equivalently the intersection of $G$ in the Pl\"ucker embedding
with the codimension $2$ linear subspace $\mathbb{P}_{k}(\Lambda^{2}V^{\vee}/L)$
of $\mathbb{P}_{k}(\Lambda^{2}V^{\vee})$. A closed point $E\subset V_{\bar{k}}$
of $G_{\bar{k}}$ belongs to $W_{L}$ if and only if $s_{\bar{k}}|_{E}=0$
for every skew-symmetric form $s$ in $L$ and, arguing as in $\S$
\ref{subsec:Hyperplane-Sec}, we see from the conormal sequence
\[
\mathcal{C}_{W_{L}/G}=\Lambda^{2}\mathcal{Q}^{\vee}|_{W_{L}}\otimes L_{W_{L}}\stackrel{d}{\to}\Omega_{G/k}^{1}|_{W_{L}}=\mathcal{H}om(\mathcal{Q}|_{W_{L}},\mathcal{S}|_{W_{L}})\to\Omega_{W_{L}/k}^{1}\to0
\]
that $W_{L,\bar{k}}$ is smooth at a closed point $E\subset V_{\bar{k}}$
if and only if $E\not\subset\mathrm{Ker}\tilde{s}_{\bar{k}}$ for
all $s\in L\setminus\{0\}$. In particular, $W_{L}$ is smooth if
and only if the line $\mathbb{P}_{k}(L^{\vee})\subset\mathbb{P}_{k}(\Lambda^{2}V)$
is contained in $\mathbb{P}_{k}(\Lambda^{2}V)\setminus\mathbb{G}_{k}(V,2)$.
This yields a natural correspondence between smooth codimension $2$
linear sections $W_{L}$ of $G$ and $k$-points of the open subset
of $\mathbb{G}_{k}(\Lambda^{2}V,2)$ parametrizing such lines, which
is a homogeneous space under the action of $\mathrm{PGL}_{5}(k)$. 
\begin{lem}
\label{lem:Unique-plane-W5}A smooth linear section $W_{L}$ of $G$
does not contain any $\sigma_{3,0}$-solid of $G$. It contains a
unique $\sigma_{2,2}$-plane $\Xi_{L}=\sigma_{2,2}(V_{3,L})$, where
$V_{3,L}\subset V$ is the linear span of the kernels of the linear
maps $\tilde{s}$, $s\in L\setminus\{0\}$. 
\end{lem}

\begin{proof}
By adjunction formula, a solid $\Pi$ contained in $W_{L}$ would
have normal sheaf isomorphic to $\mathcal{O}_{\Pi}(-1)$, and hence
$\mathcal{O}_{W_{L}}(\Pi)$ would be an invertible sheaf of degree
$1$ on $W_{L}$. This is impossible since $\mathrm{Pic}(W_{L})\cong\mathbb{Z}$
is generated by an ample invertible sheaf of degree $5$. A plane
$\sigma_{2,2}(V_{3})$ of $G$ is contained in $W_{L}$ if and only
if $V_{3}$ contains the linear span of the kernels of the linear
maps $\tilde{s}$ corresponding to the forms $s\in L\setminus\{0\}$.
The image of $\mathbb{P}_{k}(L^{\vee})\subset\mathbb{P}_{k}(\Lambda^{2}V)\setminus\mathbb{G}_{k}(V,2)$
by the morphism $\mathbb{P}_{k}(\Lambda^{2}V)\setminus\mathbb{G}_{k}(V,2)\to\mathbb{G}_{k}(V,4)$
given by the complete linear system of quadrics containing $\mathbb{G}_{k}(V,2)$
is a smooth conic $C_{L}$ in $\mathbb{G}_{k}(V,4)\cong\mathbb{P}_{k}(V^{\vee})$
whose $k$-points are the kernels $\mathrm{Ker}\tilde{s}\subset V$
of the elements $s\in L\setminus\{0\}$. Letting $V_{3,L}\subset V$
be the unique $3$-dimensional $k$-vector subspace such that $\mathbb{P}_{k}(V_{3,L}^{\vee})$
contains $C_{L}$, we conclude that $\sigma_{2,2}(V_{3,L})$ is the
unique $\sigma_{2,2}$-plane in $W_{L}$. 
\end{proof}
\vspace{-1.0em}

\subsection{\label{subsec:Projection-Sigma3-0}Projections from $\sigma_{3,0}$-solids }

Let $V_{1}\subset V$ be a $1$-dimensional $k$-vector subspace,
let $p:V\to\bar{V}=V/V_{1}$ be the quotient morphism and let $\bar{G}=\mathbb{G}_{k}(\bar{V}^{\vee},2)$
with universal sequence $0\to\mathcal{S}_{\bar{G}}\to\bar{V}_{\bar{G}}^{\vee}\to\mathcal{Q}_{\bar{G}}\to0$.
Let $Z_{\langle s\rangle}\subset G$ be a smooth hyperplane section
determined by the image $s\in\Lambda^{2}V^{\vee}$ of a symplectic
form $\bar{s}\in\Lambda^{2}\bar{V}^{\vee}$ and let $Q_{\langle\bar{s}\rangle}\subset\bar{G}$
be the zero scheme of $\bar{s}$. Let $\Pi=\sigma_{3,0}(V_{1})=G\cap\mathbb{P}_{k}(\Lambda^{2}V^{\vee}/\Lambda^{2}(\bar{V})^{\vee})$
be the solid of $G$ contained in $Z_{\langle s\rangle}$ determined
by $V_{1}$. The \emph{projection of $G$ from the solid $\Pi$ }is
the dominant rational map \begin{equation}\label{eq:Grass-Sigma-3-0-proj}
\pi_{\Pi}: G=\mathbb{G}_{k}(V^{\vee},2) \dashrightarrow \bar{G}=\mathbb{G}_{k}(\bar{V}^{\vee},2)
\end{equation} given by the restriction to $G$ of the linear projection $\mathbb{P}_{k}(\Lambda^{2}V^{\vee})\dashrightarrow\mathbb{P}_{k}(\Lambda^{2}\bar{V}^{\vee})$.
The morphism $\pi_{\Pi}|_{G\setminus\Pi}$ maps a $k$-point $E\subset V$
of $G$ not containing $V_{1}$ to the $k$-point $p(E)$ of $\bar{G}$
and conversely, the closure in $G$ of a fiber of $\pi_{\Pi}$ over
a $k$-point $\bar{E}\subset\bar{V}$ of $\bar{G}$ is the plane $\sigma_{2,2}(p^{-1}(\bar{E}))$
of $G$. 

The restriction $\pi_{\Pi}:Z_{\langle s\rangle}\dashrightarrow Q_{\langle\overline{s}\rangle}$
of $\pi_{\Pi}$ to $Z_{\langle s\rangle}$ is called the \emph{projection
of $Z_{\langle s\rangle}$ from its solid} $\Pi=\Pi_{\langle s\rangle}$.
The restriction $\pi_{\Pi}|_{Z_{\langle s\rangle}\setminus\Pi}$ maps
a $k$-point of $Z_{\langle s\rangle}$ represented by an $s$-isotropic
$k$-point $E\subset V$ of $G$ not containing $V_{1}$ to the $k$-point
of $Q_{\langle\bar{s}\rangle}$ represented by the $\bar{s}$-isotropic
$k$-point $p(E)$ of $\bar{G}$. 

To state the next result, we put $(X_{6},\mathbf{Q}_{4},\mathcal{E}_{6})=(G,\bar{G},\mathcal{Q}_{\bar{G}}^{\vee})$
and $(X_{5},\mathbf{Q}_{3},\mathcal{E}_{5})=(Z_{\langle s\rangle},Q_{\langle\bar{s}\rangle},\mathscr{S})$,
where $\mathscr{S}=\mathcal{Q}_{\bar{G}}^{\vee}|_{Q_{\langle\bar{s}\rangle}}$
is the spinor locally free sheaf of rank $2$ on the quadric threefold
$Q_{\langle\bar{s}\rangle}\subset\bar{G}$, see e.g. \cite{Ot88}. 
\begin{prop}
\label{prop:Projection-Sigma3-0 } For $i=5,6$, let $Y_{\Pi}\subset X_{i}\times\mathbf{Q}_{i-2}$
be the graph of $\pi_{\Pi}$ with projections $\mathrm{p}_{X_{i}}:Y_{\Pi}\to X_{i}$
and let $\mathrm{p}_{\mathbf{Q}_{i-2}}:Y_{\Pi}\to\mathbf{Q}_{i-2}$.
Then we have the following Sarkisov link \begin{equation}\label{eq:Grass-Sigma-3-0-link}
 \xymatrix@C4.5em@R0.95em{ \mathbb{P}_\Pi(\mathcal{C}_{\Pi/X_i}) \cong \mathbb{P}_{\mathbf{Q}_{i-2}}(\mathcal{E}_{i}) \ar[d] \ar@{^{(}->}[rr] & & Y_{\Pi} \ar[dl]_{\mathrm{p}_{X_i}} \ar[dr]^{\mathrm{p}_{\mathbf{Q}_{i-2}}} \ar[r]^-{\cong} &  \mathbb{P}_{\mathbf{Q}_{i-2}}(\mathcal{E}_i\oplus \mathcal{O}_{\mathbf{Q}_{i-2}})\ \ar[d] \\ \Pi \ar@{^{(}->}[r] &   X_i \ar@{-->}[rr]^-{\pi_{\Pi}} &  & \mathbf{Q}_{i-2}}
\end{equation}  where $\mathrm{p}_{X_{i}}:Y_{\Pi}\to X_{i}$ is the contraction of
the sub-bundle $\mathbb{P}_{\mathbf{Q}_{i-2}}(\mathcal{E}_{i})\subset\mathbb{P}_{\mathbf{Q}_{i-2}}(\mathcal{E}_{i}\oplus\mathcal{O}_{\mathbf{Q}_{i-2}})$
onto $\Pi$. 
\end{prop}

\begin{proof}
The projection $\mathrm{p}_{\bar{G}}:Y_{\Pi}\to\bar{G}$ is isomorphic
to the projective bundle $\mathbb{G}_{\bar{G}}(\mathcal{F},2)\cong\mathbb{P}_{\bar{G}}(\mathcal{F}^{\vee})\to\bar{G}$,
where $\mathcal{F}$ is the cokernel of the injective homomorphism
$\mathcal{S}_{\bar{G}}\to\bar{V}_{\bar{G}}^{\vee}\to V_{\bar{G}}^{\vee}$.
The latter is locally free of rank $3$, isomorphic to an extension
of $V_{1,\bar{G}}^{\vee}$ by $\mathcal{Q}_{\bar{G}}$, hence to $\mathcal{Q}_{\bar{G}}\oplus\mathcal{O}_{\bar{G}}$
due to the vanishing of $H^{1}(\bar{G},\mathcal{Q}_{\bar{G}})$, and
the projection $\mathrm{p}_{G}:Y_{\Pi}\to G$ contracts the projective
sub-bundle $\mathbb{P}_{\bar{G}}(\mathcal{Q^{\vee}}_{\bar{G}})\subset\mathbb{P}_{\bar{G}}(\mathcal{Q^{\vee}}_{\bar{G}}\oplus\mathcal{O}_{\bar{G}})$
to $\Pi$. This identifies in particular $\mathbb{P}_{\bar{G}}(\mathcal{Q}_{\bar{G}}^{\vee})$
with the exceptional divisor $\mathbb{P}_{\Pi}(\mathcal{C}_{\Pi/G})$
of the blow-up of $\Pi$ in $G$. The corresponding diagram for the
smooth hyperplane section $Z_{\langle s\rangle}$ follows immediately
by restriction. 
\end{proof}
\begin{example}
\label{exa:Projection-sigma30} With Notation \ref{nota:Plucker-Coord},
the kernel $V^{\bot}$ of the linear map $\tilde{s}$ associated to
the skew-symmetric form $s=e_{1}^{\vee}\wedge e_{3}^{\vee}-e_{2}^{\vee}\wedge e_{4}^{\vee}$
equals $\langle e_{5}\rangle$. The associated hyperplane section
$Z_{\langle s\rangle}=G\cap\{w_{13}-w_{24}=0\}$ is the smooth fivefold
in $\mathbb{P}_{k}^{8}\subset\mathbb{P}_{k}^{9}$ with coordinates
$w_{ij}$, $(i,j)\neq(2,4)$, defined by the equations \[\left\{\begin{array}{r} w_{12}w_{34}-w_{13}^2+w_{14}w_{23}=0\\ w_{12}w_{35}-w_{13}w_{25}+w_{15}w_{23}=0\\ w_{12}w_{45}-w_{14}w_{25}+w_{13}w_{15}=0\\ w_{13}w_{45}-w_{14}w_{35}+w_{15}w_{34}=0\\ w_{23}w_{45}-w_{13}w_{35}+w_{25}w_{34}=0 \end{array} \right.\]
Putting $\bar{V}=V/\langle e_{5}\rangle$, the image of $G$ by the
projection $$\mathbb{P}_{k}(\Lambda^{2}V^{\vee})=\mathbb{P}_{k}^{9}\dashrightarrow\mathbb{P}_{k}^{5}=\mathbb{P}_{k}(\Lambda^{2}\bar{V}^{\vee}),\,[w_{ij}]_{1\leq i<j\leq5}\mapsto[w_{12}:w_{13}:w_{14}:w_{23}:w_{24}:w_{34}]$$
from the solid $\Pi=\sigma_{3,0}(V^{\bot})=\{w_{12}=w_{13}=w_{14}=w_{23}=w_{24}=w_{34}=0\}$
is the smooth quadric fourfold $\bar{G}=\mathbb{G}_{k}(\bar{V}^{\vee},2)=\{\overline{w}_{12}\overline{w}_{34}-\overline{w}_{13}\overline{w}_{24}+\overline{w}_{14}\overline{w}_{23}=0\}$
in $\mathbb{P}_{k}^{5}$ with Pl\"ucker coordinates $\bar{w}_{ij}=\bar{e}_{i}^{\vee}\wedge\bar{e}_{j}^{\vee}$,
$1\leq i<j\leq4$, where $\bar{e}_{i}$ denotes the image of $e_{i}$
in $\bar{V}$. The image of $Z_{\langle s\rangle}$ by this projection
is the smooth quadric threefold $Q_{\langle\bar{s}\rangle}=\{\overline{w}_{12}\overline{w}_{34}-\overline{w}_{13}^{2}+\overline{w}_{14}\overline{w}_{23}=0\}$
in $\mathbb{P}_{k}^{4}\subset\mathbb{P}_{k}^{5}$ with coordinates
$\bar{w}_{ij}$, $(i,j)\neq(2,4)$. 
\end{example}

With the notation above, let $\mathrm{Sp}_{k}(\bar{V}^{\vee},\bar{s})$
be the symplectic group of the symplectic form $\bar{s}\in\Lambda^{2}\bar{V}^{\vee}$
and let $\mathrm{PSp}{}_{k}(\bar{V}^{\vee},\bar{s})$ be its image
in $\mathrm{PGL}_{k}(\bar{V}^{\vee})$. We infer the following description: 
\begin{cor}
\label{cor:Aut-Z5}There exists a split exact sequence of $k$-group
schemes
\[
0\to\mathrm{Aut}_{k}(Z_{\langle s\rangle},\Pi_{\langle s\rangle})_{0}\cong\mathbb{G}_{a,k}^{4}\rtimes\mathbb{G}_{m,k}\to\mathrm{Aut}_{k}(Z_{\langle s\rangle},\Pi_{\langle s\rangle})=\mathrm{Aut}_{k}(Z_{\langle s\rangle})\to\mathrm{Aut}_{k}(Q_{\langle s\rangle})\cong\mathrm{PSp}_{k}(\bar{V}^{\vee},\bar{s})\to0,
\]
where $\mathrm{Aut}_{k}(Z_{\langle s\rangle},\Pi_{\langle s\rangle})_{0}$
is the kernel of the restriction homomorphism $\mathrm{Aut}_{k}(Z_{\langle s\rangle},\Pi_{\langle s\rangle})\to\mathrm{Aut}_{k}(\Pi_{\langle s\rangle})$.
Moreover, up to the choice of a splitting $V\cong\bar{V}\oplus V^{\bot}$,
$\mathrm{Aut}_{k}(Z_{\langle s\rangle})$ is the image under the restriction
homomorphism $\mathrm{Aut}_{k}(G,Z_{\langle s\rangle})\to\mathrm{Aut}_{k}(Z_{\langle s\rangle})$
of the subgroup 
\[
\left\{ \left(\begin{array}{cc}
\mathrm{Sp}_{k}(\bar{V}^{\vee},\bar{s}) & \mathrm{Hom}_{k}((V^{\bot})^{\vee},\bar{V}^{\vee})\\
0 & \mathbb{G}_{m,k}
\end{array}\right)\right\} /\{\pm\mathrm{Id}\}
\]
 of $\mathrm{PGL}_{k}(V^{\vee})$ under the isomorphism $\Phi:\mathrm{PGL}_{k}(V^{\vee})\to\mathrm{Aut}_{k}(G)$
of (\ref{eq:Aut-Plucker}). 
\end{cor}

\begin{proof}
Since, by Lemma \ref{lem:Solids-and-planes-Z5}, $\Pi=\Pi_{\langle s\rangle}$
is the unique solid contained in $Z=Z_{\langle s\rangle}$, we have
$\mathrm{Aut}_{k}(Z)=\mathrm{Aut}_{k}(Z,\Pi)$. The action of $\mathrm{Aut}_{k}(Z,\Pi)$
lifts to the blow-up $\mathrm{p}_{Z}:Y\to Z$ of $\Pi$ and since
the fibers of the projection $\pi_{\Pi}:Z\dashrightarrow Q_{\langle\bar{s}\rangle}$
of $Z$ from $\Pi$ over $k$-points of $Q_{\langle\bar{s}\rangle}$
meet $\Pi$ along lines, $\mathrm{Aut}_{k}(Z,\Pi)_{0}$ equals the
kernel of the homomorphism $\mathrm{Aut}_{k}(Z,\Pi)\to\mathrm{Aut}_{k}(Q_{\langle\bar{s}\rangle})$
induced by $\mathrm{p}_{Q_{\langle\bar{s}\rangle}}:Y\to Q_{\langle\bar{s}\rangle}$.
Let $p:\mathrm{Aut}_{k}(Z,\Pi)\to B:=\mathrm{Aut}_{k}(Z,\Pi)/\mathrm{Aut}_{k}(Z,\Pi)_{0}$
be the quotient morphism and let $\gamma:B\to\mathrm{Aut}_{k}(Q_{\langle\bar{s}\rangle})$
be the induced injective homomorphism. Let $\Delta_{V^{\bot}}\subset\mathrm{GL}_{k}(V^{\vee})$
be the stabilizer of $\bar{V}^{\vee}\subset V^{\vee}$, see Notation
\ref{nota:Stabilizers-Schubert}. Let $\Delta_{V^{\bot},0}\cong\mathbb{G}_{a,k}^{4}\rtimes\mathbb{G}_{m,k}$
be its normal subgroup consisting of matrices $M(\mathrm{id_{4}},\lambda,U)$
and let $S_{V^{\bot}}$ be its subgroup consisting of matrices $M(A_{4},\pm1,0)$
with $A_{4}\in\mathrm{Sp}_{k}(\bar{V}^{\vee},\bar{s})$. The image
of $\mathrm{P}\Delta_{V^{\bot}}\subset\mathrm{PGL}_{k}(V^{\vee})$
by $\Phi:\mathrm{PGL}_{k}(V^{\vee})\to\mathrm{Aut}_{k}(G)$ is contained
in the stabilizer $\mathrm{Aut}_{k}(G,(Z,\Pi))$ of the pair $(Z,\Pi)$.
A direct verification shows that the homomorphism $j:\mathrm{P}\Delta_{V^{\bot}}\to\mathrm{Aut}_{k}(Z)$
obtained by composing with the restriction homomorphism $\mathrm{Aut}_{k}(G,Z)\to\mathrm{Aut}_{k}(Z)$
is injective and maps $\mathrm{P}\Delta_{V^{\bot},0}\cong\Delta_{V^{\bot},0}$
isomorphically onto $\mathrm{Aut}_{k}(Z,\Pi)_{0}$. Letting $\bar{q}\in\mathrm{Sym}^{2}(\Lambda^{2}\bar{V}^{\vee}/\langle\bar{s}\rangle)$
be the quadratic form associated to the symplectic form $\bar{s}$
(see $\S$ \ref{subsec:Hyperplane-Sec}), the conclusion then follows
from the fact that the restriction of the composition \[\mathrm{P}\Delta_{V^{\bot}}\stackrel{j}{\to}\mathrm{Aut}_{k}(Z,\Pi)\stackrel{p}{\to}B\stackrel{\gamma}{\to}  \mathrm{Aut}_{k}(Q_{\langle \bar{s}\rangle})=\mathrm{PO}_{k}(\Lambda^{2}\bar{V}^{\vee}/\langle s\rangle,\bar{q})=\mathrm{SO}_{k}(\Lambda^{2}\bar{V}^{\vee}/\langle s\rangle,\bar{q})\]to
the subgroup $\mathrm{P}S_{V^{\bot}}\cong\mathrm{PSp}_{k}(\bar{V}^{\vee},\bar{s})$
is an isomorphism onto its image $\mathrm{SO}_{k}(\Lambda^{2}\bar{V}^{\vee}/\langle s\rangle,\bar{q})$. 
\end{proof}

\subsection{\label{subsec:Projection-Sigma2-2}Projections from $\sigma_{2,2}$-planes }

Let $V_{3}\subset V$ be a $3$-dimensional $k$-vector subspace,
let $K=\Lambda^{2}V/\Lambda^{2}V_{3}$ and let $\langle s\rangle\subset K^{\vee}$
and $L\subset K^{\vee}$ be respectively a $1$-dimensional and a
$2$-dimensional linear subspace of skew-symmetric bilinear forms
on $V$ whose non-zero elements all have maximal rank. These data
determine respectively a plane $\Xi=\sigma_{2,2}(V_{3})=G\cap\mathbb{P}_{k}(\Lambda^{2}V_{3}^{\vee})$
of $G$, a smooth hyperplane section $Z_{\langle s\rangle}$ of $G$
and a smooth codimension $2$ linear section $W_{L}$ of $G$ which
both contain $\Xi$. 

$\bullet$ The \emph{projection of $G$ from the plane $\Xi$ }is
the birational map \begin{equation}\label{eq:Grass-Sigma-2-2-proj}
\pi_{\Xi}: G=\mathbb{G}_{k}(V^{\vee},2) \dashrightarrow \mathbb{P}_{k}(K^{\vee})
\end{equation} induced by the restriction to $G$ of the linear projection $\mathbb{P}_{k}(\Lambda^{2}V^{\vee})\dashrightarrow\mathbb{P}_{k}(K^{\vee})$.
The morphism $\pi_{\Xi}|_{G\setminus\Xi}$ maps a $k$-point $E\subset V$
of $G$ not contained in $V_{3}$ to the image of $\Lambda^{2}E\subset\Lambda^{2}V$
by the quotient homomorphism $\Lambda^{2}V\to K$. Let $Z_{\Xi}=G\cap\mathbb{P}_{k}(\Lambda^{2}V^{\vee}/\Lambda^{2}(V/V_{3})^{\vee})$
and $H_{G}=$$\mathbb{P}_{k}(K^{\vee}/\Lambda^{2}(V/V_{3})^{\vee})$
be the hyperplane sections of $G$ and $\mathbb{P}_{k}(K^{\vee})$
determined by the $1$-dimensional $k$-vector subspace $\Lambda^{2}(V/V_{3})^{\vee}$
of $K^{\vee}\subset\Lambda^{2}V^{\vee}$. Then the image $S_{G}$
of the rational map $\pi_{\Xi}|_{Z_{\Xi}}:Z_{\Xi}\dashrightarrow H_{G}$
equals that of the Segre embedding \[s_{1,1}:\mathbb{P}_{k}(V_{3}^{\vee})\times\mathbb{P}_{k}((V/V_{3})^{\vee})\hookrightarrow\mathbb{P}_{k}(V_{3}^{\vee}\otimes_{k}(V/V_{3})^{\vee})\cong\mathbb{P}_{k}(K^{\vee}/\Lambda^{2}(V/V_{3})^{\vee}).\]
A $k$-point $E\subset V$ of $G$ belongs to $Z_{\Xi}$ if and only
if $E\cap V_{3}\neq\{0\}$, and the closure in $G$ of the fiber of
$\pi_{\Xi}|_{Z_{\Xi}\setminus\Xi}$ over the image by $s_{1,1}$ of
a $k$-point $(V_{1}\subset V_{3},V_{4}/V_{3}\subset V/V_{3})$ of
$\mathbb{P}_{k}(V_{3}^{\vee})\times\mathbb{P}_{k}((V/V_{3})^{\vee})$
is the $\sigma_{3,1}$-plane $\sigma_{3,1}(V_{1}\subset V_{4})$. 

$\bullet$ Since $\langle s\rangle\subset K^{\vee}$, the projection
of $G$ from $\Xi$ restricts on $Z_{\langle s\rangle}$ to the birational
map \begin{equation}\label{eq:Grass-Sigma-2-2-proj}
\pi_{\Xi}: Z_{\langle s\rangle} \dashrightarrow \mathbb{P}_{k}(K^{\vee}/\langle s\rangle) 
\end{equation} defined by the complete linear system of hyperplane sections of $Z_{\langle s\rangle}$
containing $\Xi$, called the \emph{projection of $Z_{\langle s\rangle}$
from the plane $\Xi$.} Letting $H_{Z_{\langle s\rangle}}=H_{G}\cap\mathbb{P}_{k}(K^{\vee}/\langle s\rangle)$,
the image of the induced rational map $\pi_{\Xi}|_{Z_{\Xi}\cap Z_{\langle s\rangle}}:Z_{\Xi}\cap Z_{\langle s\rangle}\dashrightarrow H_{Z_{\langle s\rangle}}$
is the smooth cubic scroll $S_{Z_{\langle s\rangle}}=S_{G}\cap\mathbb{P}_{k}(K^{\vee}/\langle s\rangle)$
in $\mathbb{P}_{k}(K^{\vee}/\langle s\rangle)$. 

$\bullet$ Since $L\subset K^{\vee}$, the projection of $G$ from
$\Xi$ restricts on $W_{L}$ to the birational map \begin{equation}\label{eq:Grass-Sigma-2-2-proj}
\pi_{\Xi}: W_{L} \dashrightarrow \mathbb{P}_{k}(K^{\vee}/L)
\end{equation} defined by the complete linear system of hyperplane sections of $W_{L}$
containing $\Xi$, called the \emph{projection of $W_{L}$ from the
plane $\Xi$.} Putting $H_{W_{L}}=H_{G}\cap\mathbb{P}_{k}(K^{\vee}/L)$,
the image of the induced rational map $\pi_{\Xi}|_{Z_{\Xi}\cap W_{L}}:Z_{\Xi}\cap W_{L}\dashrightarrow H_{W_{L}}$is
the smooth rational cubic curve $S_{W_{L}}=S_{G}\cap\mathbb{P}_{k}(K^{\vee}/L)$
in $\mathbb{P}_{k}(K^{\vee}/L)$. 

To state the next result, we put $(X_{6},\mathbf{P}_{6})=(G,\mathbb{P}_{k}(K^{\vee}))$,
$(X_{5},\mathbf{P}_{5})=(Z_{\langle s\rangle},\mathbb{P}_{k}(K^{\vee}/\langle s\rangle))$
and $(X_{4},\mathbf{P}_{4})=(W_{L},\mathbb{P}_{k}(K^{\vee}/L))$.
We denote by $Y_{\Xi,i}\subset X_{i}\times\mathbf{P}_{i}$, the graph
of $\pi_{\Xi,i}=\pi_{\Xi}:X_{i}\dashrightarrow\mathbf{P}_{i}$, $i=4,5,6$. 
\begin{prop}
\label{prop:Projection-Sigma2-2}With the notation above, for $i=4,5,6$
we have the following Sarkisov link  \begin{equation}\label{eq:Grass-Sigma-2-2-link}
 \xymatrix@C1.7em{ \mathrm{Bl}_{S_{X_i}} H_{X_{i}}\cong \mathbb{P}_{\Xi}(\mathcal{C}_{\Xi/X_i}) \ar[d] \ar@{^{(}->}[rrr] & & & Y_{\Xi,i} \ar[dl]_{\mathrm{p}_{X_i}} \ar[dr]^{\mathrm{p}_2}  & & & \mathrm{Bl}_{\Xi}(Z_\Xi\cap X_i) \ar@{_{(}->}[lll] \ar[d]\\
\Xi \ar@{^{(}->}[r] & Z_{\Xi}\cap X_i \ar@{^{(}->}[r] & X_i \ar@{-->}@<1ex>[rr]^{\pi_{\Xi,i}} & & \mathbf{P}_i  \ar@{-->}@<1ex>[ll]^{\Phi_i} & H_{X_{i}} \ar@{_{(}->}[l] & S_{X_i} \ar@{_{(}->}[l].} 
\end{equation}   where $\mathrm{p}_{X_{i}}:Y_{\Xi,i}\to X_{i}$ is the blow-up of
$\Xi$, $\mathrm{Bl}_{\Xi}(Z_{\Xi}\cap X_{i})$ is the proper transform
of $Z_{\Xi}\cap X_{i}$, $\mathrm{p}_{2}:Y_{\Xi,i}\to\mathbf{P}_{i}$
if the blow-up of $S_{X_{i}}$ and $\mathrm{Bl}_{S_{X_{i}}}H_{X_{i}}$
is the proper transform of $H_{X_{i}}\subset\mathbf{P}_{i}$. The
birational inverse $\Phi_{i}$ of $\pi_{\Xi,i}$ is given by the complete
linear system of quadrics in $\mathbf{P}_{i}$ containing $S_{X_{i}}$. 
\end{prop}

\begin{proof}
All the properties follow from \cite{Todd30} and the description
above. 
\end{proof}
\begin{example}
\label{exa:Projection-sigma22} With Notation \ref{nota:Plucker-Coord},
let $V_{3}=\langle e_{3},e_{4},e_{5}\rangle$, let
\[
\Xi=\sigma_{2,2}(V_{3})=\{w_{12}=w_{13}=w_{14}=w_{15}=w_{23}=w_{24}=w_{25}=0\}
\]
be the associated plane of $G$ and let 
\[
\mathbb{P}_{k}(\Lambda^{2}V^{\vee})=\mathbb{P}_{k}^{9}\dashrightarrow\mathbb{P}_{k}^{6}=\mathbb{P}_{k}(K^{\vee}),\,[w_{ij}]_{1\leq i<j\leq5}\mapsto[w_{12}:w_{13}:w_{14}:w_{15}:w_{23}:w_{24}:w_{25}]
\]
be the associated linear projection. The skew-symmetric forms $s=e_{1}^{\vee}\wedge e_{3}^{\vee}-e_{2}^{\vee}\wedge e_{4}^{\vee}$
and $s'=e_{1}^{\vee}\wedge e_{4}^{\vee}-e_{2}^{\vee}\wedge e_{5}^{\vee}$
generate a subspace $L\subset\Lambda^{2}V^{\vee}$ whose non-zero
elements all have rank $4$. The associated smooth linear section
$W_{L}=G\cap\{w_{13}-w_{24}=0\}\cap\{w_{14}-w_{25}=0\}$ is the smooth
fourfold in $\mathbb{P}_{k}^{7}\subset\mathbb{P}_{k}^{9}$ with coordinates
$w_{ij}$, $(i,j)\neq(2,4),(2,5)$ defined by the equations \[ \left\{\begin{array}{r} w_{12}w_{34}-w_{13}^2+w_{14}w_{23}=0\\ w_{12}w_{35}-w_{13}w_{14}+w_{15}w_{23}=0\\ w_{12}w_{45}-w_{14}^2+w_{13}w_{15}=0\\ w_{13}w_{45}-w_{14}w_{35}+w_{15}w_{34}=0\\ w_{23}w_{45}-w_{13}w_{35}+w_{14}w_{34}=0 \end{array} \right. \]
The smooth hyperplane section $Z_{\langle s\rangle}=G\cap\{w_{13}-w_{24}=0\}$
of Example \ref{exa:Projection-sigma30} and the above smooth linear
section $W_{L}$ both contain $\Xi$. The images of $Z_{\Xi}=G\cap\{w_{12}=0\}$,
$Z_{\Xi}\cap Z_{\langle s\rangle}$ and $Z_{\Xi}\cap W_{L}$ by the
projection $\pi_{\Xi}$ and its successive restrictions are the smooth
threefold $S_{G}\cong\mathbb{P}_{k}^{1}\times\mathbb{P}_{k}^{2}$,
the smooth rational cubic surface $S_{Z_{\langle s\rangle}}\cong\mathbb{P}_{\mathbb{P}_{k}^{1}}(\mathcal{O}_{\mathbb{P}_{k}^{1}}(1)\oplus\mathcal{O}_{\mathbb{P}_{k}^{1}}(2))\cong\mathbb{F}_{1}$
and the smooth rational normal curve $S_{W_{L}}\cong\mathbb{P}_{k}^{1}$
with equations \[ \left\{\begin{array}{r}-w_{13}w_{24}+w_{14}w_{23}=0\\ -w_{13}w_{25}+w_{15}w_{23}=0\\ -w_{14}w_{25}+w_{15}w_{24}=0 \end{array} \right. \quad \left\{\begin{array}{r}-w_{13}^2+w_{14}w_{23}=0\\ -w_{13}w_{25}+w_{15}w_{23}=0\\ -w_{14}w_{25}+w_{15}w_{13}=0 \end{array} \right. \; \textrm{and} \; \left\{\begin{array}{r}-w_{13}^2+w_{14}w_{23}=0\\ -w_{13}w_{14}+w_{15}w_{23}=0\\ -w_{14}^2+w_{15}w_{13}=0 \end{array} \right.
\] in $\mathbb{P}_{k}^{5}$ , $\mathbb{P}_{k}^{4}$ and $\mathbb{P}_{k}^{3}$
respectively. 
\end{example}

\begin{rem}
By a result attributed to Weil \cite{Wei57}, the smooth varieties
$\mathbb{P}_{k}^{1}\times\mathbb{P}_{k}^{2}\subset\mathbb{P}_{k}^{5}$,
$\mathbb{F}_{1}\subset\mathbb{P}_{k}^{4}$ and the rational normal
curve $\mathbb{P}_{k}^{1}\subset\mathbb{P}_{k}^{3}$ are the only
smooth cubics which are not hypersurfaces. 
\end{rem}

For a smooth codimension $2$ linear section $W_{L}$ of $G$ with
unique $\sigma_{2,2}$-plane $\Xi_{L}=\sigma_{2,2}(V_{3,L})$, we
infer the following description of $\mathrm{Aut}_{k}(W_{L})$: 
\begin{cor}
\label{cor:Aut-W5}There exists a split exact sequence of $k$-group
schemes
\[
0\to\mathrm{Aut}_{k}(W_{L},\Xi_{L})_{0}\cong\mathbb{G}_{a,k}^{4}\rtimes\mathbb{G}_{m,k}\to\mathrm{Aut}_{k}(W_{L},\Xi_{L})=\mathrm{Aut}_{k}(W_{L})\to\mathrm{Aut}_{k}(S_{W_{L}})\cong\mathrm{PGL}_{k}(L)\to0,
\]
where $\mathrm{Aut}_{k}(W_{L},\Xi_{L})_{0}$ is the kernel of the
restriction homomorphism $\mathrm{Aut}_{k}(W_{L},\Xi_{L})\to\mathrm{Aut}_{k}(\Xi_{L})$. 
\end{cor}

\begin{proof}
Since, by Lemma \ref{lem:Unique-plane-W5}, $\Xi=\Xi_{L}$ is the
unique $\sigma_{2,2}$-plane contained in $W_{L}$ and since the intersection
$Z_{\Xi}\cap W_{L}$ is the union of all the $\sigma_{3,1}$-planes
contained in $W_{L}$, we have $\mathrm{Aut}_{k}(W_{L})=\mathrm{Aut}_{k}(W_{L},(Z_{\Xi},\Xi))$.
Since the $\sigma_{3,1}$-planes of $W_{L}$ intersect $\Xi\cong\mathbb{P}_{k}(V_{3,L})$
along the smooth conic $C_{L^{\vee}}$ dual to $C_{L}\cong\mathbb{P}_{k}(L^{\vee})\hookrightarrow\mathbb{P}_{k}(V_{3,L}^{\vee})$
(see the proof of Lemma \ref{lem:Unique-plane-W5}), the image of
the restriction homomorphism $\mathrm{Aut}_{k}(W_{L},\Xi_{L})\to\mathrm{Aut}_{k}(\Xi_{L})$
is contained in the subgroup $\mathrm{Aut}(\Xi_{L},C_{L^{\vee}})\cong\mathrm{Aut}_{k}(C_{L^{\vee}})\cong\mathrm{PGL}_{k}(L)$.
Since on the other hand the $\sigma_{3,1}$-planes of $W_{L}$ are
the closures of the fibers of $\pi_{\Xi}:W_{L}\setminus\Xi\to\mathbb{P}_{k}(K^{\vee}/L)$
over the $k$-points of $S_{W_{L}}$, the projection $\pi_{\Xi}:W_{L}\dashrightarrow\mathbb{P}_{k}(K^{\vee}/L)$
induces an isomorphism of $k$-group schemes $\mathrm{Aut}_{k}(W_{L},(Z_{\Xi},\Xi))\cong\mathrm{Aut}_{k}(\mathbb{P}_{k}(K^{\vee}/L),(H_{W_{L}},S_{W_{L}}))$
which maps $\mathrm{Aut}_{k}(W_{L},\Xi_{L})_{0}$ isomorphically onto
the kernel $\mathbb{G}_{a,k}^{4}\rtimes\mathbb{G}_{m,k}$ of the restriction
homomorphism $\mathrm{Aut}_{k}(\mathbb{P}_{k}(K^{\vee}/L),(H_{W_{L}},S_{W_{L}}))\to\mathrm{Aut}_{k}(S_{W_{L}})$.
The latter homomorphism is a split surjection, which identifies $\mathrm{Aut}_{k}(S_{W_{L}})$
with $\mathrm{Aut}_{k}(H_{W_{L}},S_{W_{L}})$. 
\end{proof}

\section{Smooth quintic del Pezzo varieties with vector group structures}

Recall that a smooth quintic del Pezzo $k$-variety of dimension $n\in\{2,\ldots6\}$
is a $k$-form $X$ of a smooth section of the Grassmannian $\mathrm{G}(2,5)\subset\mathbb{P}_{k}^{9}$
by a linear subspace of dimension $6-n$. For $n\leq3$, the automorphism
group of $X_{\bar{k}}$ is too small to allow the existence of a vector
subgroup structure on $X_{\bar{k}}$. In this section, we consider
the case of smooth quintic del Pezzo $k$-varieties of dimension $4$,
$5$ and $6$. 

\subsection{Toric vector groups structures on linear sections of $\mathrm{G}(2,5)$}

Let $V$ be a $k$-vector space of dimension $5$ and, with the notation
introduced in $\S$ \ref{subsec:Projection-Sigma2-2}, let $X_{6}=\mathbb{G}_{k}(V^{\vee},2)$,
let $V_{3}\subset V$ be a $3$-dimensional $k$-vector subspace with
associated plane $\Xi=\sigma_{2,2}(V_{3})$ of $X_{6}$ and let $K_{6}=\Lambda^{2}V/\Lambda^{2}V_{3}$.
Let $\langle s\rangle\subset K_{6}^{\vee}$ and $L\subset K_{6}^{\vee}$
be respectively a $1$-dimensional and a $2$-dimensional linear subspace
of skew-symmetric bilinear forms on $V$ whose non-zero elements all
have maximal rank. Put $K_{5}^{\vee}=K_{6}^{\vee}/\langle s\rangle$,
$K_{4}^{\vee}=K_{6}^{\vee}/L$ and let $X_{5}=Z_{\langle s\rangle}$
and $X_{4}=W_{L}$ be the smooth linear sections of $X_{6}$ containing
$\Xi$ defined by $\langle s\rangle$ and $L$ respectively. Let $F_{i}=\mathrm{Hom}_{k}(K_{i}^{\vee}/\Lambda^{2}(V/V_{3})^{\vee},\Lambda^{2}(V/V_{3})^{\vee})\cong k^{\oplus i}$,
$i=4,5,6$, and let $\mathbb{V}_{k}(F_{i}^{\vee})$ be the associated
vector group. We derive from the exact sequence
\[
0\to\Lambda^{2}(V/V_{3})^{\vee}\stackrel{a}{\to}K_{i}^{\vee}\stackrel{b}{\to}K_{i}^{\vee}/\Lambda^{2}(V/V_{3})^{\vee}\to0,\quad i=4,5,6
\]
a faithful homomorphism of $k$-group schemes $\mathbb{V}_{k}(F_{i}^{\vee})\to\mathrm{GL}_{k}(K_{i}^{\vee})$,
$f\mapsto\mathrm{id}_{K_{i}^{\vee}}+a\circ f\circ b$ corresponding
to a $\mathbb{V}_{k}(F_{i}^{\vee})$-action on $\mathbb{P}_{k}(K_{i}^{\vee})$
restricting to the trivial action on the invariant hyperplane $H_{X_{i}}=\mathbb{P}_{k}(K_{i}^{\vee}/\Lambda^{2}(V/V_{3})^{\vee})$
and having $\mathbb{P}_{k}(K_{i}^{\vee})\setminus H_{X_{i}}$ as an
open orbit. By Proposition \ref{prop:Projection-Sigma2-2}, this action
lifts through the birational projection $\pi_{\Xi}:X_{i}\dashrightarrow\mathbb{P}_{k}(K_{i}^{\vee})$
from the plane $\Xi$ to a $\mathbb{V}_{k}(F_{i}^{\vee})$-action
on $X_{i}$ with open orbit $X_{i}\setminus(Z_{\Xi}\cap X_{i})$ and
whose restriction to $\Xi$ is trivial. As a consequence, we obtain
the following: 
\begin{prop}
\label{prop:Existence-Gan-smooth-section}Every smooth section of
$\mathrm{G}(2,5)\subset\mathbb{P}_{k}^{9}$ by a linear subspace of
codimension $\leq3$ admits a vector group structure. 
\end{prop}

\begin{example}
\label{exa:Explicit-Toric-Structures}With the notation of Example
\ref{exa:Projection-sigma22}, let $V_{3}=\langle e_{3},e_{4},e_{5}\rangle$
with associated linear projection 
\[
\mathbb{P}_{k}^{9}\dashrightarrow\mathbb{P}_{k}^{6},\,[w_{ij}]_{1\leq i<j\leq5}\mapsto[w_{12}:w_{13}:w_{14}:w_{15}:w_{23}:w_{24}:w_{25}],
\]
 let $s=e_{1}^{\vee}\wedge e_{3}^{\vee}-e_{2}^{\vee}\wedge e_{4}^{\vee}\in K_{6}^{\vee}$
and $L\subset K_{6}^{\vee}$ be the subspace generated by $s$ and
$s'=e_{1}^{\vee}\wedge e_{4}^{\vee}-e_{2}^{\vee}\wedge e_{5}^{\vee}$. 

$\bullet$ For the basis $t_{ij}=(e_{i}^{\vee}\wedge e_{j}^{\vee})\otimes(\bar{e}_{1}\wedge\bar{e}_{2})$,
$i=1,2$, $j=3,4,5$ of $F_{6}^{\vee}$, the corresponding action
of $\mathbb{V}_{k}(F_{6}^{\vee})\cong\mathrm{Spec}(k[t_{ij}])$ on
$\mathbb{P}_{k}(K_{6}^{\vee})=\mathbb{P}_{k}^{6}$ with open orbit
$\mathbb{P}_{k}^{6}\setminus\{w_{12}=0\}$ is the ``toric'' $\mathbb{G}_{a}^{6}$-structure
on $\mathbb{P}_{k}^{6}$ given by $w_{12}\mapsto w_{12}$ and $w_{ij}\mapsto w_{ij}+t_{ij}w_{12}$
for $(i,j)\neq(1,2)$.\footnote{By \cite{AR17}, every complete toric variety admitting a vector group
structure has such a structure which is normalized by the torus. We
used the term ``toric'' here to indicate the fact that the given
$\mathbb{G}_{a}^{6}$-action is normalized by the toric structure
on $\mathbb{P}_{k}^{6}$. } Its lift to $X_{6}=G$ is the restriction of the $\mathbb{V}_{k}(F_{6}^{\vee})$-action
on $\mathbb{P}_{k}(\Lambda^{2}V^{\vee})$ induced by the second exterior
power of the representation 
\begin{equation}
\rho_{6}:\mathbb{V}_{k}(F_{6}^{\vee})\to\mathrm{GL}_{k}(V^{\vee}),\quad(t_{13},t_{23},t_{14},t_{24},t_{15},t_{25})\mapsto\left(\begin{array}{ccccc}
1 & 0 & -t_{23} & -t_{24} & -t_{25}\\
0 & 1 & t_{13} & t_{14} & t_{15}\\
\vdots & \ddots & 1 & 0 & 0\\
\vdots &  & \ddots & 1 & 0\\
0 & \cdots & \cdots & 0 & 1
\end{array}\right).\label{eq:Ga6-Rep-Grass}
\end{equation}
Comparing with the classification in \cite[Chapter 3, $\S$ 3, IV]{SuTy66},
the image of $\rho_{6}$ is one of the two maximal abelian unipotent
subgroups of $\mathrm{GL}_{5}$ corresponding to cases $N_{2}$ and
$N_{3}$. The other one, given by the representation dual to $\rho_{6}$,
corresponds to a vector group structure on $\mathbb{G}_{k}(V,2)\cong\mathbb{G}_{k}(V^{\vee},3)$.

$\bullet$ The corresponding action of $\mathbb{V}_{k}(F_{5}^{\vee})\cong\mathrm{Spec}(k[t_{13},t_{23},t_{14},t_{15},t_{25}])$
on $\mathbb{P}_{k}(K_{5}^{\vee})=\mathbb{P}_{k}^{5}$ with open orbit
$\mathbb{P}_{k}^{5}\setminus\{w_{12}=0\}$ is the toric $\mathbb{G}_{a}^{5}$-structure
on $\mathbb{P}_{k}^{5}$ given by $w_{12}\mapsto w_{12}$ and $w_{ij}\mapsto w_{ij}+t_{ij}w_{12}$
for $(i,j)\neq(1,2),(2,4)$. Its lift to $X_{5}=Z_{\langle s\rangle}$
is the restriction of the $\mathbb{V}_{k}(F_{5}^{\vee})$-action on
$\mathbb{P}_{k}(\Lambda^{2}V^{\vee})$ preserving $Z_{\langle s\rangle}=G\cap\{w_{13}-w_{24}=0\}$
induced by the second exterior power of the representation 
\begin{equation}
\rho_{5}:\mathbb{V}_{k}(F_{5}^{\vee})\to\mathrm{GL}_{k}(V^{\vee}),\quad(t_{13},t_{23},t_{14},t_{15},t_{25})\mapsto\left(\begin{array}{ccccc}
1 & 0 & -t_{23} & -t_{13} & -t_{25}\\
0 & 1 & t_{13} & t_{14} & t_{15}\\
\vdots & \ddots & 1 & 0 & 0\\
\vdots &  & \ddots & 1 & 0\\
0 & \cdots & \cdots & 0 & 1
\end{array}\right).\label{eq:Ga5-Rep-Z5}
\end{equation}
The stabilizer $\mathrm{Stab}(\langle e_{5}^{\vee}\rangle)$ of the
subspace $\langle e_{5}^{\vee}\rangle$ is the subgroup $\mathrm{Spec}(k[t_{15},t_{25}])$
of $\mathbb{V}_{k}(F_{5}^{\vee})$ and the induced action of the vector
group $\mathbb{V}_{k}(F_{5}^{\vee})/\mathrm{Stab}(\langle e_{5}^{\vee}\rangle)\cong\mathrm{Spec}(k[\bar{t}_{13},\bar{t}_{23},\bar{t}_{14}])$
on $\bar{V}^{\vee}=(V/\langle e_{5}\rangle)^{\vee}$ endowed with
the basis dual to that determined by the images $\bar{e}_{i}$ of
the $e_{i}$, $i=1,\ldots,4$, is given by the representation 
\begin{equation}
\bar{\rho}_{5}:\mathbb{V}_{k}(F_{5}^{\vee})/\mathrm{Stab}(\langle e_{5}^{\vee}\rangle)\to\mathrm{Sp}_{4}(\bar{V}^{\vee},\bar{s}),\quad(\bar{t}_{13},\bar{t}_{23},\bar{t}_{14})\mapsto\left(\begin{array}{cccc}
1 & 0 & -\bar{t}_{23} & -\bar{t}_{13}\\
0 & 1 & \bar{t}_{13} & \bar{t}_{14}\\
0 & 0 & 1 & 0\\
0 & 0 & 0 & 1
\end{array}\right),\label{eq:Induced-Rep-Symplectic}
\end{equation}
where $\bar{s}=\bar{e}_{1}^{\vee}\wedge\bar{e}_{3}^{\vee}-\bar{e}_{2}^{\vee}\wedge\bar{e}_{4}^{\vee}$
is the symplectic form on $\bar{V}$ induced by $s$. 

$\bullet$ The corresponding action of $\mathbb{V}_{k}(F_{4}^{\vee})\cong\mathrm{Spec}(k[t_{13},t_{23},t_{14},t_{15},])$
on $\mathbb{P}_{k}(K_{4}^{\vee})=\mathbb{P}_{k}^{4}$ with open orbit
$\mathbb{P}_{k}^{4}\setminus\{w_{12}=0\}$ is the toric $\mathbb{G}_{a}^{4}$-structure
on $\mathbb{P}_{k}^{4}$ given by $w_{12}\mapsto w_{12}$ and $w_{ij}\mapsto w_{ij}+t_{ij}w_{12}$
for $(i,j)\neq(1,2),(2,4),(2,5)$. Its lift to $X_{4}=W_{L}$ is the
restriction of the $\mathbb{V}_{k}(F_{4}^{\vee})$-action on $\mathbb{P}_{k}(\Lambda^{2}V^{\vee})$
preserving $W_{L}=G\cap\{w_{13}-w_{24}=w_{14}-w_{25}=0\}$ induced
by the second exterior power of the representation 
\begin{equation}
\rho_{4}:\mathbb{V}_{k}(F_{4}^{\vee})\to\mathrm{GL}_{k}(V^{\vee}),\quad(t_{13},t_{23},t_{14},t_{15})\mapsto\left(\begin{array}{ccccc}
1 & 0 & -t_{23} & -t_{13} & -t_{14}\\
0 & 1 & t_{13} & t_{14} & t_{15}\\
\vdots & \ddots & 1 & 0 & 0\\
\vdots &  & \ddots & 1 & 0\\
0 & \cdots & \cdots & 0 & 1
\end{array}\right).\label{eq:Ga4-Rep-W5}
\end{equation}
\end{example}

\subsection{Proof of Theorem \ref{thm:MainTh1}}

We now proceed to the proof of Theorem \ref{thm:MainTh1}, each case
$n\in\{4,5,6\}$ is treated separately in the next subsections.

\subsubsection{Proof of Theorem \ref{thm:MainTh1} for sixfolds}

A smooth quintic del Pezzo sixfold $X$ is a $k$-form of the Grassmannian
$\mathrm{G}(2,5)$. Recall \cite{BS64} that isomorphism classes of
$k$-forms of a projective $k$-variety $X$ are in one-to-one correspondence
with the elements of the Galois cohomology set $H^{1}(\Gamma,\mathrm{Aut}_{\bar{k}}(X_{\bar{k}}))$
of continuous Galois $1$-cocycles $\gamma:\Gamma=\mathrm{Gal}(\bar{k}/k)\to\mathrm{Aut}_{\bar{k}}(X_{\bar{k}})$,
where $\mathrm{Aut}_{\bar{k}}(X_{\bar{k}})$ is endowed with the discrete
topology and the natural action of $\Gamma$ by conjugation. Since
$\mathrm{Aut}_{\bar{k}}(\mathrm{G}(2,5)_{\bar{k}})\cong\mathrm{PGL}_{5}(\bar{k})=\mathrm{Aut}_{\bar{k}}(\mathbb{P}_{\bar{k}}^{4})$
(see $\S$ \ref{subsec:Grass-Prelim}), isomorphism classes of $k$-forms
of $\mathrm{G}(2,5)$ are in one-to-one correspondence with isomorphism
classes of $k$-forms $P$ of $\mathbb{P}_{k}^{4}$. In view of Lemma
\ref{lem:open-orbit}, the existence of a $k$-rational point is a
necessary condition for the existence of a vector group structure
on $X$. Conversely, for a $k$-form $X$ of $\mathrm{G}(2,5)$ containing
a $k$-rational point, the corresponding $k$-form $P$ of $\mathbb{P}_{k}^{4}$
contains a closed sub-variety $C$ defined over $k$ whose base extension
$C_{\bar{k}}$ is a line in $P_{\bar{k}}\cong\mathbb{P}_{\bar{k}}^{4}$.
Since the restriction of a canonical divisor $K_{P}$ of $P$ to $C$
has odd degree, $C$ is the trivial $k$-form of $\mathbb{P}_{k}^{1}$
and hence, $P$ is the trivial $k$-form of $\mathbb{P}_{k}^{4}$.\textcolor{purple}{{}
}This implies in turn that $X$ is isomorphic to $\mathrm{G}(2,5)$.
By Proposition \ref{prop:Existence-Gan-smooth-section}, $\mathrm{G}(2,5)$
admits at least one vector group structure. 

The following proposition completes the proof of Theorem \ref{thm:MainTh1}
in the case $n=6$. 
\begin{prop}
\label{prop:Uniqueness-Grass}The Grassmannian $\mathrm{G}(2,5)$
has a unique class of vector group structure.
\end{prop}

\begin{proof}
Write $\mathrm{G}(2,5)=\mathbb{G}_{k}(V^{\vee},2)=G$ for some $5$-dimensional
$k$-vector space $V$. Through the isomorphism $\mathrm{PGL}_{k}(V^{\vee})\cong\mathrm{Aut}_{k}(G)$
of (\ref{eq:Aut-Plucker}), a vector group structure on $G$ is given
by the projective representation associated to a faithful representation
$\rho:\mathbb{U}\to\mathrm{GL}_{k}(V^{\vee})$ of a vector group $\mathbb{U}$.
Let $V_{1}\subset V$ be a $1$-dimensional $k$-vector subspace of
$V$ that is invariant for the representation dual to $\rho$, its
existence being guaranteed by Lemma \ref{lem:Unipotent-basic-prop}(3).
Since the action of $\mathrm{GL}_{k}(V)$ on such $1$-dimensional
subspaces is transitive, up to the choice of a basis of $V$ as in
Notation \ref{nota:Plucker-Coord} and up to changing $\rho$ by its
conjugate by a suitable automorphism of $G$ we henceforth assume
without loss of generality that $V_{1}=\langle e_{5}\rangle$ and
put $\bar{V}=V/V_{1}\cong\langle e_{1},\ldots,e_{4}\rangle$. The
image of $\rho$ is then contained in the stabilizer $\Delta_{V_{1}}$
of the subspace $\bar{V}^{\vee}\subset V^{\vee}$, see Notation \ref{nota:Stabilizers-Schubert}.
Let $\mathbb{U}'$ be the kernel of the induced representation $\mathbb{U}\to\mathrm{GL}_{k}(\bar{V}^{\vee})$
and let $\bar{\rho}:\bar{\mathbb{U}}=\mathbb{U}/\mathbb{U}'\to\mathrm{GL}_{k}(\bar{V}^{\vee})$
be the induced faithful representation. With the notation of $\S$
\ref{subsec:Projection-Sigma3-0} and Example \ref{exa:Projection-sigma30},
the projection \[ \pi_{\Pi}:G\dashrightarrow\bar{G}=\{\overline{w}_{12}\overline{w}_{34}-\overline{w}_{13}\overline{w}_{24}+\overline{w}_{14}\overline{w}_{23}=0\}\subset\mathbb{P}_{k}^{5}  \]from
the solid $\Pi=\sigma_{3,0}(V_{1})$ is then $\mathbb{U}$-equivariant
for the action of $\mathbb{U}$ on $\bar{G}$ factoring through the
action of $\bar{\mathbb{U}}$ determined under the isomorphism $\mathrm{Aut}_{k}(\bar{G})\cong\mathrm{PGL}_{k}(\bar{V}^{\vee})$
by the projective representation induced by $\bar{\rho}$. The $\mathbb{U}$-action
on $G$ lifts to a vector group structure on the blow-up $Y_{\Pi}\to G$
of $\Pi$ and, by Proposition \ref{prop:adrien} applied to the induced
morphism $\mathrm{p}_{\bar{G}}:Y_{\Pi}\to\bar{G}$, the $\bar{\mathbb{U}}$-action
on $\bar{G}$ defines a vector group structure on it. In particular,
$\bar{\mathbb{U}}$ is $4$-dimensional, say $\bar{\mathbb{U}}=\mathrm{Spec}(k[\overline{t}_{13},\overline{t}_{23},\overline{t}_{14},\overline{t}_{24}])$.
By \cite{Sha09}, $\bar{G}$ admits a unique vector group structure
given up to isomorphism by the projective representation associated
to the representation
\[
\bar{\mathbb{U}}\to\mathrm{GL}_{k}(\bar{V}^{\vee}),\overline{t}=(\overline{t}_{ij})_{i=1,2,j=3,4}\mapsto A(\overline{t})=\left(\begin{array}{cccc}
1 & 0 & -\overline{t}_{23} & -\overline{t}_{24}\\
0 & 1 & \overline{t}_{13} & \overline{t}_{14}\\
0 & 0 & 1 & 0\\
0 & 0 & 0 & 1
\end{array}\right).
\]
Write $\mathbb{U}'=\mathrm{Spec}(k[s_{1},s_{2}])$, $\mathbb{U}\cong\mathrm{Spec}(k[\bar{t}][s_{1},s_{2}])\cong\bar{\mathbb{U}}\times\mathbb{U}'$
and put $u=(\overline{t},s_{1},s_{2})$. Then, with Notation \ref{nota:Stabilizers-Schubert},
the representation $\rho:\mathbb{U}\to\Delta_{V_{1}}$ lifting $\bar{\rho}:\mathbb{\bar{U}}\to\mathrm{GL}_{k}(\bar{V}^{\vee})$
has the form $u\mapsto M(u)=M(A(\bar{t}),1,{}^{t}L(u))$ for some
row matrix $L(u)=(f_{i}(u))_{i=1,\ldots4}$ of elements of $k[u]$
such that $M(u+u')=M(u)M(u')=M(u)M(u')=M(u'+u)$. By direct computation,
this identity implies that $f_{3}=f_{4}=0$ and that $f_{1}$ and
$f_{2}$ are linear polynomials. Moreover, since $\rho$ is injective,
we must have $k[\bar{t},f_{1},f_{2}]=k[\bar{t},s_{1},s_{2}]$. Denoting
$\bar{t}_{ij}$, $f_{1}$ and $f_{2}$ anew by $t_{ij}$, $-t_{25}$
and $t_{15}$ then identifies $\rho$ with the representation (\ref{eq:Ga6-Rep-Grass})
in Example \ref{exa:Explicit-Toric-Structures}.
\end{proof}

\subsubsection{Proof of Theorem \ref{thm:MainTh1} for fivefolds}

The following proposition establishes the first part of the assertion
of Theorem \ref{thm:MainTh1} for $n=5$. 
\begin{prop}
\label{prop:Z5-trivial-forms}A smooth quintic del Pezzo fivefold
is isomorphic to a smooth hyperplane section of $\mathrm{G}(2,5)$
and all these sections are isomorphic. 
\end{prop}

\begin{proof}
Let $X$ be smooth quintic del Pezzo fivefold. By \cite{Fuj81}, $X_{\bar{k}}$
is isomorphic to a smooth hyperplane section $Z_{\langle s\rangle}$
of $G=\mathbb{G}_{\bar{k}}(V^{\vee},2)$ for some $5$-dimensional
$\bar{k}$-vector space $V$. By Lemma \ref{lem:Solids-and-planes-Z5},
$X$ contains a unique $3$-dimensional sub-scheme $\Pi$ whose base
extension $\Pi_{\bar{k}}$ to $\bar{k}$ is a $\sigma_{3,0}$-solid
of $X_{\bar{k}}$. Let $\mathcal{I}_{\Pi_{\bar{k}}}\subset\mathcal{O}_{X_{\bar{k}}}$
be the ideal sheaf of $\Pi_{\bar{k}}$, let $\mathcal{O}_{X_{\bar{k}}}(1)=\Lambda^{2}\mathcal{Q}|_{X_{\bar{k}}}$
and consider the projection 
\[
\pi_{\Pi_{\bar{k}}}:X_{\bar{k}}=Z_{\langle s\rangle}\dashrightarrow Q_{\langle\bar{s}\rangle}\subset\mathbb{P}_{\bar{k}}(H^{0}(X_{\bar{k}},\mathcal{I}_{\Pi_{\bar{k}}}(1))\cong\mathbb{P}_{\bar{k}}^{4}
\]
from $\Pi_{\bar{k}}$. Since the action of the Galois group $\Gamma=\mathrm{Gal}(\bar{k}/k)$
on $X_{\bar{k}}$ maps smooth hyperplane sections of $X_{\bar{k}}$
to smooth hyperplane sections, the projective space $\mathbb{P}_{\bar{k}}(H^{0}(X_{\bar{k}},\mathcal{I}_{\Pi_{\bar{k}}}(1))$
inherits a natural continuous linear Galois action of $\Gamma$. By
Galois descent for quasi-projective varieties and rational maps between
these, the map $\pi_{\Pi_{\bar{k}}}$ thus descends to a rational
map $\pi_{\Pi}:X\dashrightarrow Q\subset P_{4}$ whose image is a
$k$-form $Q$ of $Q_{\langle\bar{s}\rangle}$ in a $k$-form $P_{4}$
of $\mathbb{P}_{\bar{k}}^{4}$. The divisor $-K_{P_{4}}-2Q$ being
defined over $k$ and of degree $1$, $P_{4}$ is the trivial form
of $\mathbb{P}_{k}^{4}$. Let $\mathrm{p}_{X}:Y\to X$ be the blow-up
of $\Pi$. Then, by Proposition \ref{prop:Projection-Sigma3-0 },
the induced morphism $\mathrm{p}_{Q}:Y\to Q$ is an \'etale locally
trivial $\mathbb{P}^{2}$-bundle whose base extension to $\bar{k}$
is isomorphic to $\mathbb{P}_{Q_{\langle s\rangle}}(\mathscr{S}\oplus\mathcal{O}_{Q_{\langle\bar{s}\rangle}})$
where $\mathscr{S}$ denotes the spinor sheaf on $Q_{\langle\bar{s}\rangle}$.
Furthermore, the restriction $\mathrm{p}_{Q}:E\to Q$ of $\mathrm{p}_{Q}$
to the exceptional divisor of $\mathrm{p}_{X}$ is an \'etale locally
trivial $\mathbb{P}^{1}$-sub-bundle of $\mathrm{p}_{Q}:Y\to Q$,
whose base extension to $\bar{k}$ is isomorphic to the sub-bundle
$\mathbb{P}_{Q_{\langle\bar{s}\rangle}}(\mathscr{S})$ of $\mathbb{P}_{Q_{\langle\bar{s}\rangle}}(\mathscr{S}\oplus\mathcal{O}_{Q_{\langle\bar{s}\rangle}})$.
Thus, considering the direct image by $\mathrm{p}_{Q}$ of the exact
sequence $0\to\mathcal{O}_{Y}\to\mathcal{O}_{Y}(E)\to\mathcal{O}_{E}(E)\to0$
on $Y$, we conclude that $\mathrm{p}_{Q}:Y\to Q$ is isomorphic to
the $\mathbb{P}^{2}$-bundle $\mathbb{P}_{Q}(\mathcal{E})$ where
$\mathcal{E}=(\mathrm{p}_{Q})_{*}\mathcal{O}_{Y}(E)$ is a locally
free sheaf of rank $3$ on $Q$. Furthermore, $\mathscr{S}_{Q}=(\mathrm{p}_{Q})_{*}\mathcal{O}_{E}(E)$
is a locally free sheaf of rank $2$ whose base extension to $\bar{k}$
is isomorphic to the spinor sheaf $\mathcal{\mathscr{S}}$ on $Q_{\bar{k}}\cong Q_{\langle\bar{s}\rangle}$.
Since $H^{0}(Q_{\langle\bar{s}\rangle},\mathscr{S}^{\vee})\cong H^{0}(Q,\mathcal{\mathscr{S}}_{Q}^{\vee})\otimes_{k}\bar{k}$
by flat base change and since $\mathcal{\mathscr{S}}^{\vee}$ is globally
generated by sections whose zero schemes are lines of $Q_{\langle\bar{s}\rangle}\subset\mathbb{P}_{\bar{k}}^{4}$
(see e.g. \cite{Ot88}), we infer that $Q$ contains a line of $P_{4}\cong\mathbb{P}_{k}^{4}$.
A smooth quadric of $\mathbb{P}_{k}^{4}$ containing a line being
unique up to isomorphism and equal to a hyperplane section of the
smooth quadric $\mathrm{G}(2,4)\subset\mathbb{P}_{k}^{5}$, we conclude
by reading the Sarkisov link of Proposition \ref{prop:Projection-Sigma2-2}
backwards that $X$ is isomorphic to a smooth hyperplane section of
$\mathrm{G}(2,5)$ over $k$. The transitivity of the action of $\mathrm{Aut}_{k}(G)(k)$
on the set of such smooth sections, see $\S$ \ref{subsec:Hyperplane-Sec},
implies that they are all isomorphic. 
\end{proof}
Since by Proposition \ref{prop:Existence-Gan-smooth-section} every
smooth hyperplane section of $\mathrm{G}(2,5)\subset\mathbb{P}_{k}^{9}$
admits a vector group structure, the following proposition completes
the proof of Theorem \ref{thm:MainTh1} in the case $n=5$. 
\begin{prop}
\label{prop:Uniqueness-Z5} A smooth hyperplane section of $\mathrm{G}(2,5)$
has a unique class of vector group structure.
\end{prop}

\begin{proof}
Write $\mathrm{G}(2,5)=\mathbb{G}_{k}(V^{\vee},2)=G$ for some $5$-dimensional
$k$-vector space $V$. The action of $\mathrm{Aut}_{k}(G)$ on the
set of smooth hyperplane sections being transitive, we are reduced
without loss of generality to prove that the smooth hyperplane section
$Z_{\langle s\rangle}$ associated to the skew-symmetric form $s=e_{1}^{\vee}\wedge e_{3}^{\vee}-e_{2}^{\vee}\wedge e_{4}^{\vee}$
with $V^{\bot}=\langle e_{5}\rangle$ considered in Example \ref{exa:Explicit-Toric-Structures}
admits a unique class of vector group structure. Under the isomorphism
of Corollary \ref{cor:Aut-Z5}, a vector group structure on $Z_{\langle s\rangle}$
is given by a certain faithful representation $\rho:\mathbb{U}\to\mathrm{GL}_{k}(V^{\vee})$
of a vector group $\mathbb{U}$ with image contained in the subgroup
of the stabilizer $\Delta_{V^{\bot}}$ of the subspace $\bar{V}^{\vee}\subset V^{\vee}$
consisting of matrices of the form
\[
M(A_{4},\lambda,U)=\left(\begin{array}{cc}
A_{4} & U\\
0 & \lambda
\end{array}\right)\textrm{ with }A_{4}\in\mathrm{Sp}_{4}(\bar{V}^{\vee},\bar{s}),\,U\in\mathrm{Hom}_{k}((V^{\bot})^{\vee},\bar{V}^{\vee}),\,\lambda\in\mathrm{GL}_{k}((V^{\bot})^{\vee})=\mathbb{G}_{m,k}.
\]
Let $\mathbb{U}'$ be the kernel of the induced representation $\mathbb{U}\to\mathrm{Sp}_{4}(\bar{V}^{\vee},\bar{s})$
and let $\bar{\rho}:\bar{\mathbb{U}}=\mathbb{U}/\mathbb{U}'\to\mathrm{Sp}_{4}(\bar{V}^{\vee},\bar{s})$
be the induced injective homomorphism. With the notation of $\S$
\ref{subsec:Projection-Sigma3-0} and Example \ref{exa:Projection-sigma30},
the projection \[ \pi_{\Pi}:Z_{\langle s\rangle} \dashrightarrow Q_{\langle\bar{s}\rangle}=\{\overline{w}_{12}\overline{w}_{34}-\overline{w}_{13}^{2}+\overline{w}_{14}\overline{w}_{23}=0\} \]from
the solid $\Pi=\sigma_{3,0}(V^{\bot})$ is then $\mathbb{U}$-equivariant
for the action of $\mathbb{U}$ on $Q_{\langle\bar{s}\rangle}$ factoring
through the action of $\bar{\mathbb{U}}$ determined under the isomorphism
$\mathrm{Aut}_{k}(Q_{\langle\bar{s}\rangle})\cong\mathrm{PSp}_{k}(\bar{V}^{\vee},\bar{s})$
by the projective representation induced by $\bar{\rho}$. The $\mathbb{U}$-action
on $Z_{\langle s\rangle}$ lifts to a vector group structure on the
blow-up $Y_{\Pi}\to Z_{\langle s\rangle}$ of $\Pi$ and, by Proposition
\ref{prop:adrien} applied to the induced morphism $\mathrm{p}_{Q_{\langle\bar{s}\rangle}}:Y_{\Pi}\to Q_{\langle\bar{s}\rangle}$,
the $\bar{\mathbb{U}}$-action on $Q_{\langle\bar{s}\rangle}$ defines
a vector group structure on it. In particular, $\bar{\mathbb{U}}$
is $3$-dimensional, say $\bar{\mathbb{U}}=\mathrm{Spec}(k[\overline{t}_{13},\overline{t}_{23},\overline{t}_{24}])$.
By \cite{Sha09}, $Q_{\langle\bar{s}\rangle}$ admits a unique vector
group structure up to isomorphism, which, with our choice of coordinates,
is induced by the representation $\bar{\rho}_{5}:\mathbb{V}_{k}(F_{5}^{\vee})/\mathrm{Stab}(\langle e_{5}^{\vee}\rangle)\to\mathrm{Sp}_{4}(\bar{V}^{\vee},\bar{s})$
of (\ref{eq:Induced-Rep-Symplectic}) in Example \ref{exa:Explicit-Toric-Structures}.
Writing $\mathbb{U}'=\mathrm{Spec}(k[s_{1},s_{2}])$, $\mathbb{U}\cong\mathrm{Spec}(k[\bar{t}][s_{1},s_{2}])\cong\bar{\mathbb{U}}\times\mathbb{U}'$
and $u=(\overline{t},s_{1},s_{2})$, the same argument as in the proof
of Proposition \ref{prop:Uniqueness-Grass} implies that the lift
$\rho:\mathbb{U}\to\Delta_{V^{\bot}}\subset\mathrm{GL}_{k}(V^{\vee})$
of $\bar{\rho}:\bar{\mathbb{U}}\to\mathrm{Sp}_{4}(\bar{V}^{\vee},\bar{s})$
has the form $u\mapsto M(u)=M(A(\bar{t}),1,{}^{t}L(s))$ where $L(s)=(f_{1}(s),f_{2}(s),0,0)$
is a row matrix of linear elements of $k[s]$ with the property that
$k[\bar{t},f_{1},f_{2}]=k[\bar{t},s_{1},s_{2}]$. Denoting the $\bar{t}_{ij}$,
$f_{1}$ and $f_{2}$ anew by $t_{ij}$, $-t_{25}$ and $t_{15}$
respectively, we get that $\rho:\mathbb{U}\to\mathrm{GL}_{k}(V^{\vee})$
is the representation (\ref{eq:Ga5-Rep-Z5}) of Example \ref{exa:Explicit-Toric-Structures}. 
\end{proof}

\subsubsection{Proof of Theorem \ref{thm:MainTh1} for fourfolds}

The following proposition completes the proof of Theorem \ref{thm:MainTh1}.
\begin{prop}
\label{prop:Uniqueness-W5}A smooth quintic del Pezzo fourfold is
isomorphic to a smooth section of $\mathrm{G}(2,5)\subset\mathbb{P}_{k}^{9}$
by a linear subspace of codimension $2$. Furthermore, all such sections
are isomorphic and admit exactly one class of vector group structure.
\end{prop}

\begin{proof}
Let $X$ be smooth quintic del Pezzo fourfold. By \cite{Fuj81}, we
can assume that $X_{\bar{k}}$ is a smooth section $W_{L}$ of $\mathrm{G}(2,5)_{\bar{k}}$
by a linear subspace of codimension $2$. By Lemma \ref{lem:Unique-plane-W5},
$X$ contains a unique $2$-dimensional sub-scheme $\Xi$ whose base
extension $\Xi_{\bar{k}}$ to $\bar{k}$ is a $\sigma_{2,2}$-plane
of $W_{L}$. Let $\mathcal{I}_{\Xi_{\bar{k}}}\subset\mathcal{O}_{X_{\bar{k}}}$
be the ideal sheaf of $\Xi_{\bar{k}}$, let $\mathcal{O}_{X_{\bar{k}}}(1)=\Lambda^{2}\mathcal{Q}|_{X_{\bar{k}}}$
and consider the projection 
\[
\pi_{\Xi_{\bar{k}}}:X_{\bar{k}}=W_{L}\dashrightarrow\mathbb{P}_{\bar{k}}(H^{0}(X_{\bar{k}},\mathcal{I}_{\Xi_{\bar{k}}}(1))\cong\mathbb{P}_{\bar{k}}^{4}
\]
from $\Xi_{\bar{k}}$ as in $\S$ \ref{subsec:Projection-Sigma2-2}.
The projective space $\mathbb{P}_{\bar{k}}(H^{0}(X_{\bar{k}},\mathcal{I}_{\Xi_{\bar{k}}}(1))$
inherits a natural continuous linear action of the Galois group $\Gamma=\mathrm{Gal}(\bar{k}/k)$,
which stabilizes the hyperplane corresponding to the unique section
of $\mathcal{I}_{\Xi_{\bar{k}}}(1)$ whose zero scheme is the special
hyperplane section $Z_{\Xi_{\bar{k}}}\cap X_{\bar{k}}$ of $X_{\bar{k}}$
spanned by the $\sigma_{3,1}$-planes of $X_{\bar{k}}$. In particular,
the divisor $Z_{\Xi_{\bar{k}}}\cap X_{\bar{k}}$ is defined over $k$,
say $Z_{\Xi_{\bar{k}}}\cap X_{\bar{k}}=(Z_{\Xi})_{\bar{k}}$ for some
geometrically irreducible divisor $Z_{\Xi}\subset X$ containing $\Xi$.
In view of Proposition \ref{prop:Projection-Sigma2-2}, the projection
$\pi_{\Xi_{\bar{k}}}$ thus descends to a birational map $\pi_{\Xi}:X\dashrightarrow P_{4}$,
with image equal to a $k$-form $P_{4}$ of $\mathbb{P}_{k}^{4}$,
which contract $Z_{\Xi}$ onto a smooth curve $C$ contained in a
hypersurface $P_{3}\subset P_{4}$ such that the triple $(P_{4},P_{3},C)$
is a $k$-form of the triple $(\mathbb{P}_{k}^{4},\mathbb{P}_{k}^{3},C_{3})$,
where $\mathbb{P}_{k}^{3}$ is a hyperplane of $\mathbb{P}_{k}^{4}$
and $C_{3}\subset\mathbb{P}_{k}^{3}$ is a smooth rational cubic curve.
Thus, $P_{4}$ and $P_{3}$ are trivial forms of $\mathbb{P}_{k}^{4}$
and $\mathbb{P}_{k}^{3}$ respectively, and, since the intersection
of $C$ with a hyperplane of $\mathbb{P}_{k}^{3}$ is a divisor of
degree $3$ on $C$, it follows that $C\cong\mathbb{P}_{k}^{1}$.
Reading the Sarkisov link of Proposition \ref{prop:Projection-Sigma2-2}
backwards, we conclude that $X$ is isomorphic to a smooth section
of $\mathrm{G}(2,5)$ by a linear subspace of codimension $2$, and
the uniqueness up to isomorphism follows from the transitivity of
the action $\mathrm{PGL}_{5}(k)$ on triples $(\mathbb{P}_{k}^{4},\mathbb{P}_{k}^{3},C_{3})$
where $C_{3}\subset\mathbb{P}_{k}^{3}$ is a smooth rational cubic.
Proposition \ref{prop:Existence-Gan-smooth-section} implies that
$X$ admits a vector group structure. To show the uniqueness, we observe
that being unique, the $\sigma_{2,2}$-plane $\Xi$ is stable under
any faithful action of a unipotent abelian group $\mathbb{U}$ on
$X$ defining a vector group structure on $X$. This vector group
structure lifts to the blow-up $\mathrm{p}_{\Xi}:Y\to X$ of $\Xi$
and then, by Proposition \ref{prop:adrien}, descends via the contraction
$\mathrm{p}_{2}:Y\to\mathbb{P}_{k}^{4}$ of the proper transform of
$Z_{\Xi}$ to a faithful $\mathbb{U}$-action on $\mathbb{P}_{k}^{4}$
defining a vector group structure, for which the pair $(\mathbb{P}_{k}^{3},C)$
is globally $\mathbb{U}$-stable. By the classification of vector
group structures on $\mathbb{P}_{k}^{4}$, see \cite{HT99} or \cite[Corollary 3.6]{HM20},
the unique class of vector group structure with this property is the
toric $\mathbb{G}_{a}^{4}$-structure on $\mathbb{P}_{k}^{4}$ described
in Example \ref{exa:Explicit-Toric-Structures}. 
\end{proof}
\begin{rem}
A by-product of Proposition \ref{prop:Uniqueness-W5} is that among
the four non-isomorphic compactifications of $\mathbb{A}_{\mathbb{C}}^{4}$
into the smooth quintic del Pezzo fourfold classified by Prokhorov
\cite[Theorem 3.1]{Pr94}, only the case (i) can be endowed with a
vector group structure making it an\emph{ equivariant} compactification
of $\mathbb{G}_{a,\mathbb{C}}^{4}$. 
\end{rem}

\begin{rem}
By \cite[$\S$ 2.2]{DK5}, smooth sections of $\mathrm{G}(2,5)\subset\mathbb{P}_{k}^{9}$
by linear subspaces of codimension $3$ have in general non-trivial
$k$-forms, whose isomorphism classes are parametrized by equivalence
classes of non-degenerate ternary quadratic forms over $k$. In contrast,
Proposition \ref{prop:Z5-trivial-forms} and Proposition \ref{prop:Uniqueness-W5}
imply that smooth del Pezzo quintics of dimension $n=4,5$ do not
have non-trivial $k$-forms. Since these are compactifications of
$\mathbb{A}_{k}^{n}$, one deduces from the techniques in \emph{loc.
cit.} that a proper morphism $f:X\to Y$ between normal varieties
over an algebraically closed field of characteristic zero whose general
closed fibers are smooth quintic del Pezzo varieties of dimension
$n=4,5$ contains a vertical $\mathbb{A}^{n}$-cylinder in the sense
of \cite{DK5}. Similarly, a proper morphism $f:X\to Y$ whose general
closed fibers are isomorphic to $\mathrm{G}(2,5)$ and which has a
rational section contains a vertical $\mathbb{A}^{6}$-cylinder. 
\end{rem}

\section{Vector group structures on terminal quintic del pezzo threefolds
and canonical quintic del Pezzo surfaces}

Classifying vector group structures on all del Pezzo quintics, including
singular ones, is a challenging problem. For instance, many families
of singular del Pezzo quintics of dimension $n=3,4,5$ endowed with
a vector group structure can be constructed as hyperplanes sections
of $G=\mathrm{G}(2,5)$ with respect to the Pl\"ucker embedding containing
a $\sigma_{2,2}$-plane $\Xi$ of $G$ (see $\S$ \ref{subsec:Projection-Sigma2-2}
for the notation) which is fixed by the $\mathbb{G}_{a,k}^{6}$-structure
on $G$ corresponding under the birational projection $\pi_{\Xi}:G=\mathrm{G}(2,5)\dashrightarrow\mathbb{P}_{k}^{6}=\mathbb{P}(H^{0}(G,\mathcal{I}_{\Xi}(1)))$
to the toric $\mathbb{G}_{a,k}^{6}$-structure on $\mathbb{P}_{k}^{6}$
described in Example \ref{exa:Explicit-Toric-Structures}. Namely,
every linear subspace $L$ of codimension $m=1,2,3$ of $\mathbb{P}_{k}^{6}$
which intersects the complement $H_{\infty}=\{w_{12}=0\}$ of the
open orbit transversely has stabilizer isomorphic to $\mathbb{G}_{a,k}^{m}$
and the induced effective action of the quotient group $\mathbb{G}_{a,k}^{6-m}$
on $L$ defines a vector group structure on it. The proper transform
of $L$ by $\pi_{\Xi}$ is then a linear section $X_{L}$ of $G$
of dimension $6-m$, singular in general, endowed with a vector structure
for which the inclusion $X_{L}\hookrightarrow G$ is equivariant for
the induced action of a subgroup $\mathbb{G}_{a,k}^{6-m}$ of the
group $\mathbb{G}_{a,k}^{6}$ which acts with an open orbit on $G$. 

Here, we consider mildly singular del Pezzo quintics of dimension
$3$ and $2$, a case of special interest due to the fact that no
smooth del Pezzo quintics in these dimensions admit a vector group
structure. \vspace{-0.5em}

\subsection{Vector group structures on terminal quintic del Pezzo threefolds}

Over algebraically closed fields of characteristic zero, non-smooth
terminal quintic del Pezzo threefolds are classified in \cite{Pr13}.
By Corollary 5.3 in \emph{loc. cit.},\emph{ }the main invariant of
such a threefold $X$ is the rank $r(X)={\rm rk}_{{\mathbb{Z}}}{\rm Cl}(X)$
of its divisor class group, which is either $2$, $3$ or $4$. Furthermore,
all the singularities of $X$ are ordinary double points and, by \cite[Corollary 8.3.1]{Pr13},
the number of such nodes equals $r(X)-1$. 
\begin{lem}
A terminal quintic del Pezzo threefold $X$ over $k$ whose base extension
$X_{\bar{k}}$ to $\bar{k}$ has one or two nodes does not admit a
vector group structure. 
\end{lem}

\begin{proof}
Since it is enough to show that $X_{\bar{k}}$ does not admit any
vector group structure, we henceforth assume that $k=\bar{k}$. Assume
that $X$ has a vector group structure. Then the latter lifts to a
vector group structure on a small $\mathbb{Q}$-factorialization $\xi:\hat{X}\to X$
of $X$. If $X$ has a unique node then, by \cite[Theorem 3.6]{JP08}
or \cite[Theorem 5.2]{Pr13}, $\hat{X}$ is isomorphic to the total
space of a $\mathbb{P}^{1}$-bundle $\pi:\hat{X}\cong\mathbb{P}_{\mathbb{P}_{k}^{2}}(\mathcal{E})\to\mathbb{P}_{k}^{2}$
for some stable, hence simple, locally free sheaf $\mathcal{E}$ of
rank $2$ on $\mathbb{P}_{k}^{2}$. We would thus obtain a vector
group structure on $\mathbb{P}_{\mathbb{P}_{k}^{2}}(\mathcal{E})$
which is impossible by Corollary \ref{cor:Simple-no-struct}. If $X$
has two nodes then, by \cite[Theorem 8.1 and Corollary 8.1.1]{Pr13},
$\hat{X}$ is isomorphic to the blow-up $\sigma:\hat{X}\to\tilde{X}$
of the total space of the projective bundle $\pi:\tilde{X}=\mathbb{P}_{\mathbb{P}_{k}^{2}}(\mathcal{T}_{\mathbb{P}_{k}^{2}}(-1))\to\mathbb{P}_{k}^{2}$
at a point. By Proposition \ref{prop:adrien}, the vector group structure
on $\hat{X}$ descends to a vector group structure on $\tilde{X}$.
But again, this is impossible since $\mathcal{T}_{\mathbb{P}_{k}^{2}}(-1)$
is a simple sheaf. 
\end{proof}
We now consider the case of terminal quintic del Pezzo threefolds
whose base extensions to $\bar{k}$ possess exactly three nodes, which
we henceforth call for short \emph{trinodal quintic del Pezzo threefolds}. 
\begin{prop}
\label{prop:trinodal-link-sigma_2,2}A trinodal quintic del Pezzo
threefold $X_{3}$ is isomorphic to a section of $\mathrm{G}(2,5)\subset\mathbb{P}_{k}^{9}$
by a linear subspace of codimension $3$. It contains a unique $\sigma_{2,2}$-plane
$\Xi$ of $\mathrm{G}(2,5)$ and the projection $\pi_{\Xi}:\mathrm{G}(2,5)\dashrightarrow\mathbb{P}_{k}^{6}$
from $\Xi$ induces the following Sarkisov link \begin{equation}\label{eq:V5-Sigma-2-2-link}
 \xymatrix@C1.7em{ \mathrm{Bl}_{S_{X_3}} H_{X_3} \ar[d] \ar@{^{(}->}[rr]  & & \mathrm{Bl}_{\Xi} X_3\cong \mathrm{Bl}_{S_{X_3}}\mathbb{P}^3_k \ar[dl]_{\mathrm{p}_{X_3}} \ar[dr]^{\mathrm{p}_2} \\
\Xi \ar@{^{(}->}[r] & X_3 \ar@{-->}@<1ex>[rr]^{\pi_{\Xi}} & & \mathbb{P}^3_k \ar@{-->}@<1ex>[ll]^{\Phi} & H_{X_{3}} \ar@{_{(}->}[l] & S_{X_3} \ar@{_{(}->}[l].} 
\end{equation}   where $\mathrm{p}_{X_{3}}:\mathrm{Bl}_{\Xi}X_{3}\to X_{3}$ is the
blow-up of $\Xi$, $\mathrm{p}_{2}:\mathrm{Bl}_{\Xi}X_{3}\to\mathbb{P}_{k}^{3}$
is the blow-up of a smooth $0$-dimensional sub-scheme $S_{X_{3}}$
of length $3$ of $\mathbb{P}_{k}^{3}$ not contained in a line, $H_{X_{3}}$
is the unique hyperplane of $\mathbb{P}_{k}^{3}$ containing $S_{X_{3}}$
and $\mathrm{Bl}_{S_{X_{3}}}H_{X_{3}}$ is its proper transform. Furthermore,
the birational inverse $\Phi$ of $\pi_{\Xi}$ is given by the complete
linear system of quadrics of $\mathbb{P}_{k}^{3}$ containing $S_{X_{3}}$. 
\end{prop}

\begin{proof}
By \cite[Theorem 7.1]{Pr13}, $X_{3,\bar{k}}$ is isomorphic to a
threefold obtained from $\mathbb{P}_{\bar{k}}^{3}$ as the blow-up
$\sigma:Y\to\mathbb{P}_{\bar{k}}^{3}$ of three non-colinear closed
points, say $p_{1},p_{2},p_{3}$, followed by the contraction $\xi:Y\to X_{3,\bar{k}}$
of the proper transforms by $\sigma$ of the lines in $\mathbb{P}_{\bar{k}}^{3}$
passing through $p_{i}$ and $p_{j}$, $1\leq i<j\leq3$ to the nodes
of $X_{3,\bar{k}}$. The class group of $X_{3,\bar{k}}$ is freely
generated by the classes of the proper transform $\tilde{H}$ of the
unique hyperplane $H\subset\mathbb{P}_{\bar{k}}^{3}$ containing the
points $p_{i}$ and of the images $F_{i}$, $1\leq i\leq3$, of the
exceptional divisors of $\sigma$. By \cite[Theorem 7.2]{Pr13}, $P$
and the $F_{i}$ are the only planes contained in $X_{3,\bar{k}}$.
Their union is thus defined over $k$ and since $P$ is the only plane
among these which fully contains the singular locus of $X_{3,\bar{k}}$
it is defined over $k$ as well, say $P=\Xi_{\bar{k}}$ for some closed
sub-scheme $\Xi$ of $X_{3}$. This implies in turn that the union
of the $F_{i}$ is defined over $k$, say $\bigcup F_{i}=F_{\bar{k}}$
for some closed sub-scheme $F$ of $X_{3}$. Since the divisor $(\Xi+F)_{\bar{k}}=P+\sum_{i=1}^{3}F_{i}$
is linearly equivalent to the proper transform in $X_{3,\bar{k}}$
of any hyperplane $H'\subset\mathbb{P}_{\bar{k}}^{3}$ not passing
through the points blown-up, it is Cartier. The invertible sheaf $\mathcal{O}_{X_{3,\bar{k}}}(P+\sum_{i=1}^{3}F_{i})$
is then the base extension to $\bar{k}$ of the invertible sheaf $\mathcal{O}_{X_{3}}(1):=\mathcal{O}_{X_{3}}(\Xi+F)$
and, letting $\mathcal{I}\subset\mathcal{O}_{X_{3}}$ be the ideal
sheaf of the singular locus of $X_{3}$, the rational map $\sigma\circ\xi^{-1}:X_{3,\bar{k}}\dashrightarrow\mathbb{P}_{\bar{k}}^{3}$
is the base extension of the birational map $\pi:X_{3}\dashrightarrow\mathbb{P}(H^{0}(X_{3},\mathcal{I}(1))\cong\mathbb{P}_{k}^{3}$.
The latter maps $\Xi$ to a hyperplane $H_{X_{3}}\subset\mathbb{P}_{k}^{3}$
and contracts $F$ to a smooth closed sub-scheme $S_{X_{3}}\subset H_{X_{3}}$
of length $3$ whose base extension to $\bar{k}$ equals the union
of the points $p_{i}$. Since $S_{X_{3}}$ is not contained in a line,
there exists a smooth rational cubic curve $S_{X_{4}}\subset\mathbb{P}_{k}^{3}$
whose scheme-theoretic intersection with $H_{X_{3}}$ equals $S_{X_{3}}$.
Considering $\mathbb{P}_{k}^{3}$ as a hyperplane $H_{X_{4}}$ of
$\mathbb{P}_{k}^{4}$, it follows from Proposition \ref{prop:Projection-Sigma2-2}
that the image of the rational map $\mathbb{P}_{k}^{4}\dashrightarrow\mathbb{P}_{k}^{7}$
given by the complete linear system of quadrics containing $S_{X_{4}}\subset H_{X_{4}}$
is a smooth quintic del Pezzo fourfold $X_{4}$ containing $X_{3}$
as a hyperplane section and which has the proper transform $\Xi$
of $H_{X_{4}}$ as its unique $\sigma_{2,2}$-plane. By construction,
the birational map $\pi:X_{3}\dashrightarrow\mathbb{P}_{k}^{3}$ then
coincides with the restriction to $X_{3}$ of the projection $\pi_{\Xi}:X_{4}\dashrightarrow\mathbb{P}_{k}^{4}$
from $\Xi$, which completes the proof.
\end{proof}
\begin{cor}
\label{cor:Iso-class-Trinodal-forms}Isomorphism classes of trinodal
quintic del Pezzo threefolds are in one-to-one correspondence with
$\mathrm{PGL}_{2}(k)$-orbits of smooth $0$-dimensional sub-schemes
of $\mathbb{P}_{k}^{1}$ of length three. Furthermore, every such
threefold admits a unique class of vector group structure.
\end{cor}

\begin{proof}
By Proposition \ref{prop:trinodal-link-sigma_2,2}, two trinodal quintic
del Pezzo threefolds $X_{3}$ and $X_{3}'$ are isomorphic if and
only if there exists an automorphism of $\mathbb{P}_{k}^{3}$ which
maps the pair $(H_{X_{3}},S_{X_{3}})$ onto the pair $(H_{X_{3}'},S_{X_{3}'})$.
Being of length $3$ and not contained in a line, the schemes $S_{X_{3}}$
and $S_{X_{3}'}$ are contained in smooth $k$-rational conics $C_{X_{3}}$
and $C_{X_{3}'}$ of $H_{X_{3}}$ and $H_{X_{3}'}$, respectively.
Since $\mathrm{Aut}_{k}(\mathbb{P}_{k}^{3})$ acts transitively on
pairs $(H,C)$ consisting of a hyperplane $H$ of $\mathbb{P}_{k}^{3}$
and a smooth $k$-rational conic $C\cong\mathbb{P}_{k}^{1}$ in it
and since for such pairs the restriction homomorphism $\mathrm{Aut}_{k}(H,C)\to\mathrm{Aut}_{k}(C)$
is an isomorphism, we conclude that $X_{3}$ and $X_{3}'$ are isomorphic
if and only if there exists an isomorphism $\varphi:C_{X_{3}}\to C_{X_{3}'}$
which maps $S_{X_{3}}$ onto $S_{X_{3}'}$. This holds if and only
if $S_{X_{3}}$ and $\varphi^{-1}(S_{X_{3}'})$ belong to the same
orbit of the action of $\mathrm{Aut}_{k}(C_{X_{3}})(k)\cong\mathrm{\mathrm{PGL}_{2}}(k)$. 

For the second assertion, since the singular locus of $X_{3}$ and
the unique $\sigma_{2,2}$-plane $\Xi$ of $X_{3}$ are stable under
any vector group action on $X_{3}$, it follows from Proposition \ref{prop:adrien}
applied to the birational morphism $\mathrm{Bl}_{\Xi}X\to\mathbb{P}_{k}^{3}$
that the Sarkisov link of Proposition \ref{prop:trinodal-link-sigma_2,2}
is equivariant for any vector group structure on $X_{3}$ and that
the corresponding vector group structure on $\mathbb{P}_{k}^{3}$
stabilizes the non-linear closed sub-scheme $S_{X_{3}}$ of length
$3$. By the classification \cite{HT99} of vector group structures
on $\mathbb{P}_{k}^{3}=\mathrm{Proj}_{k}(k[x_{0},x_{1},x_{2},x_{3}])$,
the unique class of vector group structure with this property is that
of the toric $\mathbb{G}_{a}^{3}$-structure defined by $x_{0}\mapsto x_{0}$
and $x_{i}\mapsto x_{i}+t_{i}x_{0}$, $1\leq i\leq3$. Conversely,
this structure lifts to a vector group structure on the blow-up of
$S_{X_{3}}$, which, by Proposition \ref{prop:adrien} again, descends
in turn to a vector group structure on $X_{3}$. 
\end{proof}
\begin{rem}
\label{rem:singular}By Proposition \ref{prop:trinodal-link-sigma_2,2}
and Corollary \ref{cor:Iso-class-Trinodal-forms}, every trinodal
quintic del Pezzo threefold contains the affine $3$-space $\mathbb{A}_{k}^{3}$
as a Zariski open subset. In contrast, there exist in general $k$-forms
of smooth quintic del Pezzo threefolds contains which do not contain
$\mathbb{A}_{k}^{3}$, see \cite[Theorem 12]{DK5}. 
\end{rem}

\begin{example}
With Notation \ref{nota:Plucker-Coord}, let $V_{3}=\langle e_{3},e_{4},e_{5}\rangle$
and let $\Xi=\sigma_{2,2}(V_{3})$ be the associated plane of $G=\mathrm{G}(2,5)\subset\mathbb{P}_{k}^{9}$
as in Example \ref{exa:Projection-sigma22}. For every $\beta\in k^{*}$,
the linear section 
\[
X_{3}(\beta)=G\cap\{w_{13}-w_{24}=0\}\cap\{\beta w_{14}-w_{25}=0\}\cap\{\beta w_{15}+w_{23}=0\}
\]
is a trinodal quintic del Pezzo threefold containing $\Xi$, isomorphic
to the sub-variety in $\mathbb{P}_{k}^{6}$ with coordinates $w_{ij}$,
$(i,j)\neq(2,3),(2,4),(2,5)$ defined by the equations 
\[
\left\{ \begin{array}{r}
w_{12}w_{34}-w_{13}^{2}-\beta w_{14}w_{15}=0\\
w_{12}w_{35}-\beta w_{13}w_{14}-\beta w_{15}^{2}=0\\
w_{12}w_{45}-\beta w_{14}^{2}+w_{13}w_{15}=0\\
w_{13}w_{45}-w_{14}w_{35}+w_{15}w_{34}=0\\
-\beta w_{15}w_{45}-w_{13}w_{35}+\beta w_{14}w_{34}=0
\end{array}\right.
\]
Its singular locus $\mathrm{Sing}(X_{3}(\beta))$ is the closed sub-scheme
of $\Xi\cong\mathrm{Proj}_{k}(k[w_{34},w_{35},w_{45}])$ with equations
\[
\beta w_{34}w_{45}-w_{35}^{2}=w_{34}w_{35}-\beta w_{45}^{2}=w_{35}w_{45}-w_{34}^{2}=0.
\]
Letting $\lambda,\epsilon\in\bar{k}$ be respectively a third root
of $\beta$ and a primitive third root of unity, the singular locus
of $X_{3}(\beta)_{\bar{k}}$ is the union of the three closed points
$[\lambda\epsilon^{m}:(\lambda\epsilon^{m})^{2}:1]$, $0\leq m\leq2$,
of $\Xi$. Thus, according to whether $\beta$ is cube in $k^{*}$
or not and $k^{*}$contains a primitive primitive third root of unity
or not, $\mathrm{Sing}(X_{3}(\beta))$ consists either of a single
closed point, or the union of a $k$-point and a single other closed
point, or the union of three $k$-points. The image of $\Xi$ by the
restriction $X_{3}(\beta)\dashrightarrow\mathbb{P}_{k}^{3}=\mathrm{Proj}_{k}(k[w_{12},w_{13},w_{14},w_{15}])$
of the projection from $\Xi$ is the hyperplane $H_{X_{3}(\beta)}=\{w_{12}=0\}$.
The associated smooth $0$-dimensional sub-scheme $S_{X_{3}(\beta)}$
of length $3$ is the closed sub-scheme of $H_{X_{3}(\beta)}$ defined
by the equations 
\[
\beta w_{14}w_{15}+w_{13}^{2}=w_{13}w_{14}+w_{15}^{2}=w_{13}w_{15}-\beta w_{14}^{2}=0.
\]
The restriction to $X_{3}(\beta)$ of the action of the sub-group
$\mathbb{U}\cong\mathbb{G}_{a,k}^{3}$ of the vector group $\mathbb{V}_{k}(F_{6}^{\vee})$
of (\ref{eq:Ga6-Rep-Grass}) in Example \ref{exa:Explicit-Toric-Structures}
defined by $t_{24}=t_{13}$, $t_{25}=\beta t_{14}$ and $t_{23}=-\beta t_{15}$
induces a vector group structure 

\[
\left\{ \begin{array}{l}
w_{12}\longmapsto w_{12}\\
w_{13}\longmapsto w_{13}+t_{13}w_{12}\\
w_{14}\longmapsto w_{14}+t_{14}w_{12}\\
w_{15}\longmapsto w_{15}+t_{15}w_{12}\\
w_{34}\longmapsto w_{34}+2t_{13}w_{13}+\beta(t_{15}w_{14}+t_{14}w_{15})+(t_{13}^{2}+\beta t_{14}t_{15})w_{12}\\
w_{35}\longmapsto w_{35}+\beta(t_{14}w_{13}+t_{13}w_{14}+2t_{15}w_{15})+\beta(t_{13}t_{14}+t_{15}^{2})w_{12}\\
w_{45}\longmapsto w_{45}+(2\beta t_{14}w_{14}-t_{15}w_{13}-t_{13}w_{15})+(\beta t_{14}^{2}-t_{15}t_{13})w_{12}
\end{array}\right.
\]
on $X_{3}(\beta)$ with open orbit $X_{3}(\beta)\setminus\{w_{12}=0\}$. 
\end{example}

\subsection{Vector group structures on canonical quintic del Pezzo surfaces }

Del Pezzo surfaces with canonical singularities admitting a vector
group structure are classified in \cite{DL10}, see also \cite{MaSt20}\textcolor{purple}{.
}For the sake of completeness, we record the following consequence
of the classification in the quintic case: 
\begin{prop}
\label{prop:delPezzoSurfs}Up to isomorphism, there exist two quintic
del Pezzo surfaces with canonical singularities which admit a vector
group structure:

a) A surface $S$ with an $A_{3}$-singularity, whose neutral component
$\mathrm{Aut}^{0}(S)$ of the automorphism group is isomorphic to
$\mathbb{G}_{a,k}^{2}\rtimes\mathbb{G}_{m,k}$ and which admits a
unique class of vector group structure. 

b) A surface $S'$ with an $A_{4}$-singularity, whose neutral component
$\mathrm{Aut}^{0}(S')$ of the automorphism group is isomorphic to
$\mathbb{U}_{3}\rtimes\mathbb{G}_{m,k}$, where $\mathbb{U}_{3}$
is a maximal unipotent subgroup of $\mathrm{PGL}_{3}(k)$, and which
admits exactly two classes of vector group structures. 
\end{prop}

\begin{proof}
All the properties but those concerning the actual number of equivalence
classes of vector group structures are established in \cite{DL10}
and for the description of the automorphism groups in \cite[Table 1]{MaSt20},
cases $5E$ and $5F$ for $S$ and $S'$ respectively. We briefly
recall the principle of the argument in \emph{loc. cit.} and explain
how derive from it the equivalence classes of vector group structures.
Equivalence classes of vector group structures on a del Pezzo surface
$S$ with canonical singularities are in one-to-one correspondence
with those on its minimal desingularization $\tilde{S}\to S$, which
is obtained from $S$ by performing a finite sequence of successive
blow-ups of singular loci of intermediate surfaces and normalizations.
Indeed, a vector group structure stabilizes singular loci, hence canonically
lifts to their blow-ups and, by universal property, canonically lifts
as well to normalizations. Conversely, Proposition \ref{prop:adrien}
ensures that every vector group structure on $\tilde{S}$ descend
to $S$. Here, the surfaces $\tilde{S}$ are weak del Pezzo surfaces
of degree $5$ whose base extensions to $\bar{k}$ are obtained from
$\mathbb{P}_{\bar{k}}^{2}$ by performing certain finite sequences
of blow-ups of closed points. In the two cases under consideration,
the respective dual graphs of the unions of the $(-1)$-curves and
$(-2)$-curves in $\tilde{S}_{\bar{k}}$ have the following structure
$$ 
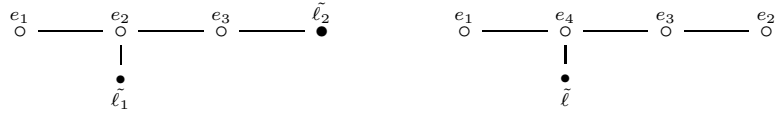
\xymatrix@R=0.6em{ \stackrel{e_1}{\circ}\ar@{-}[r]& \stackrel{e_2}{\circ} \ar@{-}[r]\ar@{-}[d]& \stackrel{e_3}{\circ} \ar@{-}[r]& \stackrel{\tilde{\ell_2}}{\bullet} \\ & \stackrel{\bullet}{\scriptstyle{\tilde{\ell_1}}}  } \qquad  \qquad \xymatrix@R=0.8em{ \stackrel{e_1}{\circ}\ar@{-}[r]& \stackrel{e_4}{\circ} \ar@{-}[r]\ar@{-}[d]&\stackrel{e_3}{\circ}\ar@{-}[r]& \stackrel{e_2}{\circ} \\ & \stackrel{\bullet}{\scriptstyle{\tilde{\ell}}} } $$
in which the vertices $\bullet$ and $\circ$ correspond respectively
to $(-1)$-curves which are the proper transforms of the lines in
$S_{\bar{k}}$ and to $(-2)$-curves which are the exceptional divisors
of the desingularization $\tilde{S}_{\bar{k}}\to S_{\bar{k}}$. Since
these diagrams have no symmetries, all the curves displayed are defined
over $k$, corresponding to irreducible smooth $k$-rational curves
in $\tilde{S}$ with the same self-intersection numbers. 

In the case of an $A_{3}$-singularity, the successive contractions
of $\tilde{\ell}_{2}$, and then of the exceptional divisors $e_{3}$,
$e_{2}$ and $e_{1}$ yield a birational morphism $\sigma:\tilde{S}\to\mathbb{P}_{k}^{2}=\mathrm{Proj}_{k}(k[u_{0},u_{1},u_{2}])$
which maps $\tilde{\ell}_{1}$ onto a line $\ell\subset\mathbb{P}_{k}^{2}$
and contracts $\tilde{\ell}_{2}\cup e_{1}\cup e_{2}\cup e_{3}$ onto
a $k$-point $p\in\ell$ , say, up to composition by a suitable automorphism
of $\mathbb{P}_{k}^{2}$, $\ell=\{u_{2}=0\}$ and $p=[1:0:0]$. A
vector group structure on $S$ and its canonical lift to $\tilde{S}$
being given, Proposition \ref{prop:adrien} implies the existence
of a unique vector group structure on $\mathbb{P}_{k}^{2}$ for which
$\sigma:\tilde{S}\to\mathbb{P}_{k}^{2}$ is equivariant. The latter
stabilizes $\ell$ as well as the proper and infinitely near base
points of $\sigma^{-1}$. By \cite[Proposition 3.2]{HT99} there are
two classes of vector group structures on $\mathbb{P}_{k}^{2}$ fixing
$\ell$ and $p$: the ``toric'' structure given by the $\mathbb{G}_{a,k}^{2}$-action
$[u_{0}:u_{1}:u_{2}]\mapsto[u_{0}+t_{0}u_{2}:u_{1}+t_{1}u_{2}:u_{2}]$
and the ``non-toric'' one given by the $\mathbb{G}_{a,k}^{2}$-action
$[u_{0}:u_{1}:u_{2}]\mapsto[u_{0}+t_{1}u_{1}+(\frac{1}{2}t_{1}^{2}+t_{0})u_{2}:u_{1}+t_{1}u_{2}:u_{2}]$.
A direct verification shows that the lift of the toric structure to
the surface $\tilde{S}_{1}$ obtained from $\tilde{S}$ by contracting
$\tilde{\ell}_{2}$ acts transitively on $e_{3}\setminus e_{2}$,
hence that this structure cannot be induced by a vector group structure
on $\tilde{S}$. On the other hand, the lift to $\tilde{S}_{1}$ of
the other structure fixes $e_{3}$ point wise, hence is descended
via the contraction of $\tilde{\ell}_{2}$ from a vector group structure
on $\tilde{S}$. Thus, $\tilde{S}$, whence $S$, has a unique class
of vector group structure. 

In the case of an $A_{4}$-singularity, the successive contractions
of $\tilde{\ell}$, and then of the exceptional divisors $e_{4}$,
$e_{3}$ and $e_{2}$ yield birational morphism $\sigma:\tilde{S}\to\mathbb{P}_{k}^{2}$
which maps $e_{1}$ onto a line $\ell\subset\mathbb{P}_{k}^{2}$ and
$\tilde{\ell}\cup e_{4}\cup e_{3}\cup e_{2}$ onto a $k$-point $p\in\ell$.
Up to composing by a suitable automorphism of $\mathbb{P}_{k}^{2}$
as above, we again infer that a vector group structure on $\tilde{S}$
is equivalent to the lift via $\sigma$ of one of the two equivalence
classes of such structures on $\mathbb{P}_{k}^{2}$ described above.
Noting that for both structures the first three points blown-up by
$\sigma$ are fixed and that the lifts of these two structures to
the resulting surface both fix $e_{4}$ point wise, we conclude that
both structures lift to $\tilde{S}$. These two structures descend
in turn on $S$, showing that $S$ has at most two equivalence classes
of vector group structures. The conclusion follows from the observation
that two so-constructed induced structures have non-isomorphic fixed
point schemes, hence are not equivalent. 
\end{proof}
\begin{rem}
With the notation of the proof of Proposition \ref{prop:delPezzoSurfs},
in the case of the del Pezzo surface $S$ with an $A_{3}$-singularity,
the contractions of $\tilde{\ell}_{1}$, $e_{1}$, $e_{2}$ and $\tilde{\ell}_{2}$
yield another birational morphism $\sigma':\tilde{S}\to\mathbb{P}_{k}^{2}$
which maps $e_{3}$ onto a line $\ell'$ and contracts $\tilde{\ell}_{1}\cup e_{2}\cup e_{1}$
and $\tilde{\ell}_{2}$ onto a pair of disctinct $k$-points of $\ell$.
In contrast with the morphism $\sigma:\tilde{S}\to\mathbb{P}_{k}^{2}$
constructed in the proof of Proposition \ref{prop:delPezzoSurfs}
which is equivariant with respect to the non-toric $\mathbb{G}_{a,k}^{2}$-structure
on $\mathbb{P}_{k}^{2}$, the birational morphism $\sigma'$ is equivariant
with respect to the toric $\mathbb{G}_{a,k}^{2}$-structure on $\mathbb{P}_{k}^{2}$.
The toric and non-toric structure on $\mathbb{P}_{k}^{2}$ thus become
equivalent on $S$ and hence, are birationally conjugated on $\mathbb{P}_{k}^{2}$
by the birational automorphism $\sigma'\circ\sigma^{-1}$. 
\end{rem}

\bibliographystyle{amsplain} 

\end{document}